\documentclass[preprint]{amsart}
\usepackage{amsmath, graphicx, amssymb, amsthm, tikz-cd}
\usepackage[color=pink,textsize=footnotesize]{todonotes}

\newtheorem{thm}{Theorem}[section]
\newtheorem{lemma}[thm]{Lemma}
\newtheorem{prop}[thm]{Proposition}
\newtheorem{conj}[thm]{Conjecture}
\newtheorem{cor}[thm]{Corollary}

\theoremstyle{definition}
\newtheorem{defn}[thm]{Definition}
\newtheorem{nota}[thm]{Notation}
\newtheorem{remark}[thm]{Remark}
\newtheorem{example}[thm]{Example}

\def\sign{\mathop{\rm sign}}
\def\Star{\mathop{\rm star}}
\def\supp{\mathop{\rm supp}}

\def\Bary{\mathop{\rm Bary}}
\def\Cx{\mathop{\rm Cx}}
\def\Hom{\mathop{\rm Hom}}
\def\rank{\mathop{\rm rank}}
\def\row{\mathop{\rm row}}

\def\Gr{\mathop{\rm Gr}}
\def\EGr{\mathrm{EGr}}
\def\MacP{\mathop{\rm MacP}}
\def\Flag{\mathop{\rm Flag}}
\def\OMFlag{\mathop{\rm OMFlag}}

\def\ll{\mathrm{lin}}
\def\aa{\mathrm{aff}}

\def\A{{\mathcal A}}

\def\DD{{\mathcal D}}

\def\H{{\mathcal H}}
\def\M{{\mathcal M}}
\def\N{{\mathcal N}}
\def\OO{{\mathcal O}}
\def\PPP{{\mathcal P}}

\def\R{{\mathbb R}}

\def\U{{\mathcal U}}
\def\V{{\mathcal V}}
\def\Y{{\mathcal Y}}
\def\Z{{\mathcal Z}}

\def\a{ a} 
\def\2{{2}}
\def\e{{\mathbf e}}
\def\s{{\mathbf s}}
\def\vv{{\mathbf v}}
\def\w{{\mathbf w}}
\def\x{{\mathbf x}}
\def\y{{\mathbf y}}
\def\z{{\mathbf z}}
\begin{document}

\title[Combinatorial formula for Pontrjagin classes]{On Gelfand and MacPherson's combinatorial formula for Pontrjagin classes}

\author{Olakunle Abawonse}
\address{Department of mathematics and statistics, Binghamton University, Binghamton, NY 13902-6000, USA.}
\curraddr{Department of mathematics, Northeastern University}
\email{o.abawonse@northeastern.edu} 
\author{Laura Anderson}
\address{Department of mathematics and statistics, Binghamton University, Binghamton, NY 13902-6000, USA.}
\email{laura@math.binghamton.edu}
\begin{abstract} In~\cite{GM} Gelfand and MacPherson gave a local and combinatorial formula for the Pontrjagin classes of a differential manifold. We give an expanded version of their discussion and highlight the origins of combinatorial differential manifolds in their work.
\end{abstract}
\maketitle

\section{Introduction}

In 1992 Gelfand and MacPherson (\cite{GM}) found a local and combinatorial formula for the Pontrjagin classes of a compact smooth manifold by introducing a new combinatorial approach to smooth structures. Their approach, in 5 pages, encompassed a formidable array of techniques,  and proofs were omitted or sketched. Since then some elements of their approach have gotten further exposition and development (\cite{MacP}, \cite{bundles}, \cite{mythesis}),  but their full proof has never been presented in a broadly accessible form.

Our aim here is to provide a full and reasonably self-contained presentation of the results in~\cite{GM}. We expect few readers to come with knowledge of all of the background mathematics involved, so we have included appendices outlininging such knowledge.
In the course of preparing this we found some minor errors in the original paper, which we have listed in Appendix~\ref{app:oops}. 

 \section{Overview}

Throughout, $M$ denotes a compact smooth manifold. 

There are two main elements of Gelfand and MacPherson's approach. The first is an alternative formulation of Chern-Weil theory that gives a formula for the inverse rational Pontrjagin classes of a smooth bundle in terms of the Chern class of an associated circle bundle. The second is the development of a combinatorial model for differential manifolds, within which we can find a combinatorial formula for this Chern class, when the smooth bundle in question is the tangent bundle.

Chern-Weil theory gives formulas for characteristic classes of smooth bundles in terms of curvature forms. This may seem like an unlikely place to look for combinatorics, but read on. Given a smooth real vector bundle $\xi:E\to M$ of even rank $a$ with a Riemannian metric, we consider the associated {\em Grassmannian $2$-plane bundle} $\pi:\Y\to M$. Here
\[\Y=\{(m, V): m\in M, \mbox{ $V$ a $2$-dimensional subspace of $\xi^{-1}(m)$}\}\]
and $\pi$ is projection onto the first coordinate. This is a fiber bundle with fiber the Grassmannian $\Gr(2, \R^a)$. 
There is a canonical rank 2 real vector bundle $\xi_\2:E_\2\to\Y$, where $E_\2=\{(m,V,v):(m,V)\in\Y, v\in V\}$.
%
%
%
The alternative view of Chern-Weil theory in~\cite{GM} gives a formula for the inverse Pontrjagin classes of $\xi$ in terms of a curvature form on $\xi_\2$ (see Theorem~\ref{thm:chernweil}.)

Stepping back from the smooth world, we can interpret this formula (Corollary~\ref{cor:chernweil1} ) in terms of the Chern class of $\xi_\2$, or equivalently, the Chern class of the circle bundle associated to $\xi_\2$ (with coefficients twisted by the orientation presheaf). {\em Thus much of  the task of finding the rational Pontrjagin classes of a smooth bundle reduces to finding the Chern class of a circle bundle.} Finding the Chern class of a circle bundle is a more tempting combinatorial target.

The second major element of  Gelfand and MacPherson's approach is  {\em combinatorial differential manifolds}.
A combinatorial differential manifold is a simplicial complex $X$, which plays the role of the base space of a manifold, a cell decomposition $\hat X$ of $X$, and an association of an {\em oriented matroid} to each cell in $\hat X$, which plays the role of a coordinate chart at that cell. From a smooth triangulation $X\to M$ of a smooth manifold $M$ we obtain a combinatorial differential manifold. The combinatorial formula for Pontrjagin classes then comes from combinatorially mimicking the construction of $\xi_\2$ to get a simplicial $S^1$ qusifibration, and then showing that the Chern class of this quasifibration leads to a formula for the Pontrjagin classes of $M$.

We should note that  Gelfand and MacPherson's formula is not quite ``combinatorial" in the sense of clearly  depending only on the combinatorics of $X$. The smooth structure on $M$ is used to construct the combinatorial differential manifold over $X$. The final formula is in terms of a {\em fixing cycle} derived from the combinatorial differential structure, and it's not clear that the fixing cycle is independent of the combinatorial differential structure.

\section{History}
 
Thom~(\cite{Thom})  and Rohlin and \v{S}varc~(\cite{RS})  independently showed, by nonconstructive arguments, that the rational Pontrjagin classes of a smooth manifold are combinatorial invariants: that is, if $M$ is a smoothing of a simplicial complex $X$ then $X$ determines the rational Pontrjagin classes of $M$. Levitt and Rourke~\cite{LR} showed, by a nonconstructive argument, the existence of a  formula for the Pontrjagin classes that is local (i.e., given by a cocycle whose value on an open set in $X$ depends only on the combinatorics of $X$ in that set).

 Gelfand and MacPherson's approach grew out of work of Gabrielov, Gelfand, and Losik (\cite{GGL1}, \cite{GGL2}, \cite{MGGL}) that gave a formula for the first Pontrjagin class and suggested a route to all of the Pontrjagin classes. A key advance in Gelfand and MacPherson's approach was the use of oriented matroids as combinatorial analogs to tangent spaces.
 The resulting notion of {\em combinatorial differential manifolds}, and its generalization {\em matroid bundles}, was expanded on in~\cite{MacP}, \cite{bundles}, \cite{mythesis}. 
 
Questions on characteristic classes associated to combinatorial differential manifolds and matroid bundles are closely tied to questions on the topology of the {\em combinatorial Grassmannians} described in Sections~\ref{sec:MacP} and~\ref{sec:eric}, in particular the {\em MacPhersonian}. Since Gelfand and MacPherson's work, some progress has been made on elucidating the topology of the MacPhersonian, including its mod 2 cohomology (\cite{AD}), and its first few homotopy groups (\cite{macppi1}). The conjecture that $\Gr(r,\R^S)\simeq\|\MacP(r,S)\|$ remains open. (An attempt at a proof in 2009 (\cite{Biss}) was withdrawn.)

\section{An alternate form of Chern-Weil theory\label{sec:chernweil}}

Chern-Weil theory is the study of characteristic classes within the context of de Rham cohomology. It computes rational characteristic classes of smooth manifolds, and more generally principal bundles, in terms of connections and curvature.

Let $\xi:E\to B$ be a smooth bundle with a connection, and let $\Omega$ be the associated curvature form. Classical Chern-Weil theory (cf.\  \S 5.4 in~\cite{MoritaDF})) says that the total Pontrjagin class of $E$ with coefficients in $\R$ is
\[P(E)=\sum_i p_i(E)=\left[\det\left( I +\frac{1}{2\pi}\Omega\right)\right]\]

Since $p_0(E)=1$,  $P(E)$ has an inverse $\tilde{P}(E)=\sum_i\tilde{p}_i(E)$ in $H^*(B)$. Gelfand and MacPherson's alternate form of Chern-Weil theory gives a formula for $\tilde{P}$ rather than $P$, and it is in terms of the curvature form of a rank 2 bundle associated to $E$ rather than the original bundle.

Throughout this section  $\xi:E\to M$ denotes a real vector bundle  over a smooth compact manifold $M$, and $\pi:\Y\to M$ denotes the associated Grassmannian $2$-plane bundle. For most of Section~\ref{sec:chernweil} we assume that $a$ is even: this saves us many issues around orientability. The modifications needed when $a$ is odd are outlined in Section~\ref{sec:gettingbeyondsimplifications}.

\subsection{Orientability of $\pi$}\label{sec:Yorient}

\begin{prop} \label{prop:yorientable}  Let $a$ be even, and let $\xi:E\to M$ be a rank $a$ real vector bundle. Then the associated Grassmannian 2-plane bundle $\pi:\Y\to M$ is orientable.
\end{prop}

\begin{nota} For a matrix $N$ with $a$ columns and and $1\leq x<y\leq a$, let $N_{x,y}$ be the submatrix consisting of the columns of $N$ indexed by $x$ and $y$. Let $N\backslash\{x,y\}$ be the matrix obtained by deleting the columns indexed by $x$ and $y$.
\end{nota}

\begin{defn} A matrix $N$ is {\em $\{x,y\}$-reduced} if $N_{x,y}=I$.
\end{defn}

We begin by describing coordinate charts on $\Gr(2,\R^{a})$. For each $x<y$ let $W_{x,y}$ be the subset of $\Gr(2,\R^{a})$ consisting of spaces of the form $\row(N)$, where $N$ is $\{x,y\}$-reduced. A coordinate chart on $W_{x,y}$ is given by the map  $\Phi_{x,y}:\R^{2(a-2)}\to W_{x,y}$  sending $(n_1, \ldots, n_{a(a-2)})$ to the $\{x,y\}$-reduced matrix $N$ with
\[N\backslash\{x,y\}=\begin{pmatrix}
n_1&n_3&\cdots& n_{2(a-2)-1}\\
n_2&n_4&\cdots&n_{2(a-2)}
\end{pmatrix}\]
An element of  $W_{w,x}\cap W_{y,z}$ can be written as the row space of a $\{w,x\}$-reduced matrix $N$ and as the row space of a $\{y,z\}$-reduced matrix $N_{y,z}^{-1}N$. (The intersection $W_{w,x}\cap W_{y,z}$ is exactly the set of all $\row(N)$ such that $N$ is $ \{w,x\}$-reduced and $N_{y,z}$ is invertible.) This gives us the transition function $g_{(w,x),(y,z)}$ between the two coordinate systems on $W_{w,x}\cap W_{y,z}$.

\begin{lemma}\label{lem:chartsonG} If $a$ is even then $g_{(w,x),(y,z)}$ is orientation preserving.
	
	In particular, $\Gr(2,\R^{a})$ is orientable.
\end{lemma}

\begin{proof} We abbreviate $g_{(w,x),(y,z)}$ by $g$. It suffices to show that the Jacobian matrix $(\frac{\partial g_i}{\partial m_j})_{i,j}$ has positive determinant.
	For simplicity we assume $(w,x)=(1,2)$ and $\{y,z\}\subseteq\{1,2,3,4\}$.

	For a $\{1,2\}$ reduced matrix $N$ we have
	\[
	N=\begin{pmatrix}
	1&0&n_1&n_3&\cdots&n_{2(a-2)-1}\\
	0&1&n_2&n_4&\cdots&n_{2(a-2)}
	\end{pmatrix}\]
	and so
	$N_{y,x}^{-1}N$ has columns
	\[
	N_{y,x}^{-1}\begin{pmatrix}1\\0\end{pmatrix}, N_{y,x}^{-1}\begin{pmatrix}0\\1\end{pmatrix},  N_{y,x}^{-1}\begin{pmatrix}n_1\\n_2\end{pmatrix},\ldots,N_{y,x}^{-1}\begin{pmatrix}n_{2(a-2)-1}\\n_{2(a-2)}\end{pmatrix}
	\]
		We abuse notation by writing a vector of length $2(a-2)$ as a vector of $a-2$ vectors of length 2.  Thus
	\[g\left(\begin{pmatrix} n_1\\n_2\end{pmatrix},\begin{pmatrix} n_3\\n_4\end{pmatrix},\cdots, \begin{pmatrix} n_{2(a-2)-1}\\n_{2(a-2)}\end{pmatrix}\right)=
	\left(\begin{pmatrix} b_1\\b_2\end{pmatrix},\begin{pmatrix} b_3\\b_4\end{pmatrix},N_{y,z}^{-1}\begin{pmatrix} n_5\\n_6\end{pmatrix}
	\cdots, N_{y,z}^{-1}\begin{pmatrix} n_{2(a-2)-1}\\n_{2(a-2)}\end{pmatrix}\right)\]
	where the values $b_1, b_2, b_3, b_4$ depend on our values of $y$ and $z$. Thus the Jacobian of $g$ has the form
	\[\begin{pmatrix}
	J_1&0\\J_2&J_3
	\end{pmatrix}\]
	where $J_1$ is $4\times4$ and $J_3$ is a block-diagonal matrix whose $a-4$ blocks are all $N_{y,z}^{-1}$. Thus the Jacobian has determinant $|J_1||N_{y,z}|^{a-4}$. Since $a$ is even, we need only check that $|J_1|$ is positive. This is a tedious but easy check for each of the possible values of $\{y,z\}$.
\end{proof}

\begin{proof}[Proof of Proposition~\ref{prop:yorientable}] Of course, we may assume $a>2$. Consider a cover $\{U_i\}$ of $M$ and local coordinates $\phi_i:U_i\times\R^{a}\to \xi^{-1}(U_i)$. The identification $\phi_i$  gives us an identification of $\pi^{-1}(U_i)$ with $U_i\times \Gr(2, \R^{a})$. Let $p\in U_i\cap U_j$. The transition map from $U_i$ coordinates to $U_j$ coordinates on $\xi^{-1}(p)$   is given by right multiplication by a matrix $A$. Thus the transition map on $\pi^{-1}(p)$ takes each $\row(N)$ to $\row(NA)$.
	
	We wish to show that this map is orientation preserving. The matrix $A$ is a product of elementary matrices, so it's enough to show that when $A$ is an elementary matrix then this map is orientation preserving. For each $x<y$ let $W_{x,y}$ be the subset of $\pi^{-1}(p)$ consisting of elements that, in the $U_i$ trivialization, have the form $\row(N)$ with $N$ $\{x,y\}$-reduced, and let $W'_{x,y}$ be the similar set with respect to the $U_j$-trivialization. Thanks to Lemma~\ref{lem:chartsonG}, it is enough to find some nonempty $W_{w,x}\cap W'_{y,z}$ such that the transition map between these charts is orientation-preserving.
	
	Case 1: Let  $A$ be a permutation matrix, and let $\sigma$ be the corresponding permutation. Then  $NA$ is
	$\{\sigma(1), \sigma(2)\}$-reduced. and  $W_{1,2}=W'_{\sigma(1), \sigma(2)}$. Also, $\sigma$ induces a permutation $\sigma'$ taking the columns of $N\backslash\{1,2\}$ to the columns of $(NA)\backslash \{\sigma(1), \sigma(2)\}$ for each $N\in W_{1,2}$. Let $\tilde A$ be the matrix of $\sigma'$. The transition map from $W_{1,2}$ to $W'_{\sigma(1),\sigma(2)}$ is given by the matrix
	\[Q^{-1}\begin{pmatrix} \tilde A&0\\0&\tilde A\end{pmatrix}Q\]
	where $Q$ is the permutation matrix for the permutation $1,3,5,\ldots,2,4,6,\ldots$. This matrix has positive determinant $|\tilde A|^2$, and so the transition map is orientation preserving.
	
	Case 2: Let $A$ be an invertible diagonal matrix, with diagonal entries $d_1,\ldots,d_{a}$. Then $W_{1,2}=W'_{1,2}$, and the transition matrix between these charts is a diagonal matrix with diagonal entries $\frac{1}{d_1}, \frac{1}{d_2}, \frac{d_3}{d_1}, \frac{d_3}{d_2}, \ldots,\frac{d_{a}}{d_1}, \frac{d_{a}}{d_2}$. The product of these entries is a square, so is positive.
	
	Case 3: Let $A$ be a matrix with 1's on the diagonal and a single nonzero entry off the diagonal. Then there is an $x<y$ such that $W_{x.y}=W'_{x,y}$ and the transition matrix between them is a triangular matrix with 1's on the diagonal, and so this transition map is orientation-preserving.

\end{proof}

\subsection{Rank 2 Grassmannians}

There are two canonical bundles over $\Gr(2, \R^\a)$. 
\begin{itemize}
\item Let $\EGr_{2}=\{(V,x): x\in V\in\Gr(2,\R^\a)\}$ and $\tau_{2}:\EGr_{2}\to\Gr(2,\R^\a)$ be the projection map.  
\item Let $\EGr_{(a-2)}=\{(V,x):  V\in\Gr(2,\R^\a), x\in  V^\perp\}$ and $\tau_{(a-2)}:\EGr_{(a-2)}\to\Gr(\2,\R^\a)$ be the projection map.
\end{itemize}
 The product $\EGr_{(a-2)}\oplus \EGr_\2$ is a trivial bundle and the total Pontrjagin class is multiplicative, and so $P(\EGr_{(a-2)})=\tilde{P}(\EGr_\2)$. 
%
%
%
%
With respect to the orientation on the Grassmannian given in Section~\ref{sec:Yorient} we have a formula for the Poincare dual of the top Pontrjagin class of $\EGr_{(a-2)}$.

\begin{lemma}\label{lem:p1}	If $a$ is even then
	$p_{\frac{a-2}{2}}(\EGr_{(a-2)})\frown[\Gr(2,\R^\a)]=(-1)^{\frac{\a-2}{2}}.$
\end{lemma}

\begin{proof} 
	If $\xi:E\to B$ is  an oriented rank $2k$ vector bundle then 
	its $k$th Pontrjagin class is the square of its Euler class  
	(cf.\  Corollary 15.8 in~\cite{MS}). Also, Euler classes are multiplicative, so $e(E)^2=e(E\oplus E)$. Thus in our case we have $p_{\frac{a-2}{2}}(\EGr_{(a-2)})=e(\EGr_{(a-2)}\oplus \EGr_{(a-2)})$.
	
 The Poincare dual of $e(\EGr_{(a-2)}\oplus \EGr_{(a-2)})$ can be given as the sum of the number of zeroes of a section, counted with multiplicity (Appendix~\ref{app:euler}).
	Define a smooth section $s: \Gr(2, \mathbb{R}^a) \rightarrow \EGr_{(a-2)} \oplus \EGr_{(a-2)}$ by $s(V) = (V,\pi_V(\e_{a-1}), \pi_V(\e_a)))$, where $\pi_V$ denotes projection onto $V^\perp$.   The section $s$ coincides with the zero section only at $V_0 := \row(\mathbf{0}|I)\in W_{a-1,a}$, and so $e(\EGr_{(a-2)}\oplus \EGr_{(a-2)})$ is the multiplicity of $V_0$.
	
	 Let $\tilde{s}$ denote the composition, using the local coordinates on $W_{a-1,a}$ described in Section~\ref{sec:Yorient},
	\[\R^{2(a-2)}\cong W_{a-1,a} \xrightarrow{s} (\EGr_{(a-2)}\oplus \EGr_{(a-2)})|{W_{a-1,a}}\xrightarrow{\cong} W_{a-1,a}\times V_0^\perp\times V_0^\perp \rightarrow \R^{a-2} \times \R^{a-2}  \]
	The multiplicity of the intersection point $V_0$ is 1 if  $\tilde{s}_*:T_{0}\R^{2(a-2)}
		\to T_{(0,0)}(\R^{a-2}\times \R^{a-2})$ is orientation preserving and $-1$ if $\tilde{s}_*$ is orientation-reversing. The map $(\EGr_{(a-2)}\oplus \EGr_{(a-2)})|_{W_{a-1,a}}\xrightarrow{\cong} W_{a-1,a}\times \R^{a-2}\times \R^{a-2}$ takes each $(x, v_1, v_2)$ to $(x, v_1',v_2')$, where $v_i'$ is obtained from $v_i$ by deleting the last two components.
	
	The tangent space $T_{0}\R^{2(a-2)}$ has as standard basis the tangents  to the curves $t\e_i$ in $\R^{2(a-2)}$. Composing these curves with our map $\tilde{s}$, we see \[\tilde{s}(t\e_i)=\begin{cases}
	(\frac{-t}{1+t^2}\e_\frac{i+1}{2},0)&\mbox{if $i$ is odd}\\
	(0,	\frac{-t}{1+t^2}\e_\frac{i}{2})&\mbox{if $i$ is even}
	\end{cases}\]
	Thus 
	\[\tilde{s}_*(\e_i)=\begin{cases}
	(-\e_\frac{i+1}{2},0)&\mbox{if $i$ is odd}\\
	(0,-\e_\frac{i}{2})&\mbox{if $i$ is even}
	\end{cases}
	\]
	Thus the standard positively oriented ordered basis of $T_{(0,\ldots,0)}\R^{2(a-2)}$ is mapped by $\tilde{s}_*$ to the an ordered basis with the same sign as $((\e_1, 0), (0,\e_1), (\e_2,0), (0, \e_2),\ldots, (\e_{a-2},0),(0,\e_{a-2}))$. The sign of the permutation taking this to the standard basis 
$((\e_1, 0),
\ldots, (\e_{a-2},0),(0,\e_1),  
\ldots, (0,\e_{a-2}))$ of $T(\R^{a-2}\times\R^{a-2})$ is $(-1)^{\frac{(a-3)(a-2)}{2}}=(-1)^{\frac{(a-2)}{2}}$.
	
\end{proof}

A slightly different approach to this proof is taken in the proof of Lemma~\ref{lem:p1odd}.

\subsection{The alternate formula}
Recall from the Introduction the bundle $\xi_\2:E_\2\to\Y$. Let $E_{(a-2)}=\{(m,V,v):(m,V)\in\Y, v\in V^\perp\}$ and $\xi_{(a-2)}:E_{(a-2)}\to\Y$ be projection. The pullback $\pi^*(E)$ splits as $\xi_2\oplus\xi_{(a-2)}$.
A connection on $E$ pulls back to a connection on 
$\pi^* E\cong E_{\2}\oplus E_{(a-2)}$, which can be given  with respect to local coordinates by
\[\omega = \begin{bmatrix}
\omega_1  & 0\\
0 & \omega_2
\end{bmatrix}\]\\
where $\omega_1, \omega_2$ are connection matrices from  metric connections for $\xi_\2$  and $\xi_{(\a-2)}$ respectively. By choosing a trivialization and  an orthonormal frame, $\omega_1$ can be chosen to be a skew-symmetric matrix. So, \[\omega_1 = \begin{bmatrix} 0 & \omega_{12}\\
-\omega_{12} & 0
\end{bmatrix}.\]
The corresponding curvature matrix is
\[\Omega_1 = \begin{bmatrix}

0 & \Omega_{12}\\ 
-\Omega_{12} & 0

\end{bmatrix}\]
where $\Omega_{12} = d\omega_{12}$ .

\begin{lemma} \label{lem:rank2pi1}
	$p_1(E_{2})=[(\frac{1}{2\pi}\Omega_{12})^2]$.
\end{lemma}
\begin{proof}

	This follows  from the definition of the total Pontrjagin class  as   $[\det ( I + \frac{1}{2\pi}\Omega_1)]$ .
	
\end{proof}

\begin{thm}\label{thm:chernweil} For every $i$ 
	\begin{align}\label{eqn:cw}
	\tilde{p_i}(E)=\pi^!\left( \left[(-1)^i(\frac{1}{2\pi}\Omega_{12})^{\a-2+2i})\right]\right).
	\end{align}
\end{thm}

Here $\pi^!$ denotes the transfer map (integration along the fiber; see Appendix~\ref{app:transfer} and Section~\ref{sec:gettingbeyondsimplifications}).

\begin{remark} Theorem~\ref{thm:chernweil} holds for bundles of both odd and even dimension. As previously noted, we are presenting a proof for bundles of even dimension and postponing the odd-dimensional case to Section~\ref{sec:gettingbeyondsimplifications}. When $a$ is odd then $[(\frac{1}{2\pi}\Omega_{12})]$ is interpreted in the cohomology of $\Y$ with coefficients twisted by the orientation of $\xi_2$.
\end{remark}

\begin{lemma} \label{lem:pistar} If $a$ is even then $\pi^!(p_\frac{a-2}{2}(E_{(a-2)}))=(-1)^\frac{a-2}{2}$.
\end{lemma}

\begin{proof} $\pi^!(p_\frac{a-2}{2}(E_{(a-2)}))$ is the integral over the fiber of $p_\frac{a-2}{2}(E_{(a-2)})$, and by Lemma~\ref{lem:p1} this is $(-1)^\frac{a-2}{2}$.
\end{proof}

\begin{proof}[Proof of Theorem~\ref{thm:chernweil}]	\begin{align*}
	\sum_{i } \left[(-1)^i(\frac{1}{2\pi}\Omega_{12})^{ 2i}\right]&=\sum_{i} (-1)^i(p_1(E_\2))^i&\mbox{by Lemma~\ref{lem:rank2pi1}}\\
	&=(1+p_1(E_\2))^{-1}\\
	&=\tilde{P}(E_\2)
	\end{align*}
	
	Thus 	
	\begin{align*}
	P(E)\smile \pi^!\left(\sum_{i } \left[(-1)^i(\frac{1}{2\pi}\Omega_{12})^{ 2i}\right]\right) & =  \pi^!\left((\pi^*P(E))\smile \tilde{P}(E_\2) \right)&\mbox{by Lemma~\ref{lem:projection}}\\
	& =   \pi^!\left(P(\pi^*
	E)\smile \tilde{P}(E_\2)\right)&\mbox{by naturality of $P$}\\
	&=\pi^!\left(P(E_{(\a-2)}))\smile P(E_\2)\smile  \tilde{P}(E_\2)\right)&\mbox{by multiplicativity of $P$}\\
	&=\pi^!(P(E_{(\a-2)}))&\\
	\end{align*}\\
	The fiber dimension of $\pi$ is $2(a-2)$, and so $\pi^!(\alpha)=0$ for each $\alpha$ of degree less than $2(a-2)$. The only term of $P(E_{(\a-2)})$ of degree at least $2(a-2)$ is $p_{\frac{a-2}{2}}(E_{(a-2)})$, and so 
	
	\begin{align*}
	P(E)\smile \pi^!\left(\sum_{i } \left[(-1)^i(\frac{1}{2\pi}\Omega_{12})^{ 2i}\right]\right) &=\pi^!(p_{\frac{a-2}{2}}(E_{(a-2)}))\\
	&=(-1)^{\frac{a-2}{2}} &\mbox{by Lemma~\ref{lem:pistar} }
	\end{align*}
	
	Thus $\tilde{P}(E)=(-1)^{\frac{a-2}{2}}\pi^!\left(\sum_{i } \left[(-1)^i(\frac{1}{2\pi}\Omega_{12})^{ 2i}\right]\right) $. Comparing terms of degree $4i$, we see $\tilde{p}_i(E)=(-1)^{\frac{a-2}{2}}\pi^!(\left[(-1)^{i+\frac{a-2}{2}}(\frac{1}{2\pi}\Omega_{12})^{ 2(i+\frac{a-2}{2})}\right])=\pi^!(\left[ (-1)^i(\frac{1}{2\pi}\Omega_{12})^{\a-2+2i}\right])$.
\end{proof}
%

Now consider the cohomology of $\Y$ with coefficients twisted by the orientation presheaf $\OO$ of $E_{2}$. As described in Appendix~\ref{app:ss}, the class $c_1(E_\2,\OO):=[\frac{1}{2\pi}\Omega_{12}]$ is the twisted Chern class of $E_\2$. That is, it's  the image of the fiber orientation class under the differential $d_2 : E_2^{0,1} \rightarrow E_2^{2,0}$ in the Leray-Serre spectral sequence of $\rho$. (Any even power of $c_1(E_\2,\OO)$ is a class in cohomology with constant coefficients).
Thus from Theorem~\ref{thm:chernweil} we have the following.
\begin{cor} \label{cor:chernweil1} For every $i$
\[\tilde p_i(E)=\pi^!((-1)^i c_1(E_\2,\OO)^{a-2+2i})\]
\end{cor}

Now we're ready to preview the combinatorial formula. We need to interpret the elements of the above formula in terms that make sense for appropriate simplicial complexes. 

Let $\rho:\Z\to\Y$ be the circle bundle associated to $E_\2$. The Chern class $c_1(\Z,\OO)$ is the same as $c_1(E_\2,\OO)$. 

The formula which we'll combinatorialize is a formula for the Poincare dual of $\tilde p_i(E)$. 
\begin{cor}\label{cor:smoothform} $\tilde p_i(E)\frown[M]=\pi_*((-1)^ic_1(\Z,\OO)^{a-2+2i}\frown[\Y])$
\end{cor}

\begin{remark} If $M$ is orientable then so is $\Y$, and the expressions $[M]$ and $[\Y]$ in Corollary~\ref{cor:smoothform} represent the usual fundamental class. If $M$ is not orientable, then $[M]$ and $[\Y]$ denote the fundamental classes in homology with coefficients twisted by the orientation presheafs. Let $\DD$ denote the orientation presheaf of $M$. Because $\pi$ is orientable, the orientation presheaf of $\Y$ is just the pullback $\pi^*(\DD)$.
\end{remark}

\begin{remark} As discussed in Appendix~\ref{app:coefficients}, $c_1(\Z,\OO)^{a-2+2i}$ and the formula of Corollary~\ref{cor:smoothform} can be interpreted as expressions in rational cohomology.
\end{remark}

\begin{proof}[Proof of Corollary~\ref{cor:smoothform}]
\begin{align*}
\tilde p_i(E)\frown[M]&=\pi^!((-1)^ic_1(\Z,\OO)^{a-2+2i})\frown[M]&\mbox{ by Corollary~\ref{cor:chernweil1}}\\
&=\pi_*((-1)^ic_1(\Z,\OO)^{a-2+2i}\frown[\Y])&\mbox{by Proposition~\ref{prop:starvscap}}
\end{align*}
\end{proof}

\subsection{Bundles of odd dimension}\label{sec:gettingbeyondsimplifications}

Here we sketch the modifications needed when the dimension $a$ of $\xi:E\to M$ is odd. Our first problem is that $\Gr(2,\R^a)$ is not orientable.  
Let $\OO$ denote the orientation presheaf of $\tau_2:\EGr_2\to\Gr(2,\R^a)$. (This is the pullback of the presheaf we denoted by $\OO$ in the previous section.)
Recall that $\DD$ is the orientation presheaf on $M$.
The proofs of Lemma~\ref{lem:chartsonG} and Proposition~\ref{prop:yorientable} can be adapted to prove the following.

\begin{lemma} \label{lem:chartsonGodda} The orientation presheaf of $\Gr(2, \R^a)$ is $\OO^{\otimes(a-2)}$.
\end{lemma}

\begin{prop} The orientation presheaf of $\xi_{(a-2)}$  is $\OO^{\otimes(a-2)}\otimes\pi^*\DD$.
\end{prop}

Thus
we get a transfer map $\pi^!:H^*(\Y,\OO^{\otimes(a-2)})\to H^*(M,\R)$ as a composition
\[H^*(\Y,\OO^{\otimes(a-2)})\stackrel{\frown[\Y]}{\to}H_*(\Y,\pi^*\DD)\stackrel{\pi_*}{\to}H_*(M,\DD)
\stackrel{PD_M^{-1}}{\to}H^*(M,\R)\]

The analog to Lemma~\ref{lem:p1}
 is the following.
\begin{lemma}\label{lem:p1odd} If $a$ is odd then $p_{\lfloor\frac{a-2}{2}\rfloor}(\EGr_{(a-2)})=(-1)^{\lfloor\frac{a-2}{2}\rfloor}[\frac{1}{2\pi}\Omega_{12}]^{2\lfloor\frac{a-2}{2}\rfloor}$, and  $(p_{\lfloor\frac{a-2}{2}\rfloor}(\EGr_{(a-2)})[\frac{1}{2\pi}\Omega_{12}])\cap[\Gr(2,\R^a)]=(-1)^{\lfloor\frac{a-2}{2}\rfloor}$
\end{lemma}

\begin{proof} 
%
Because $\EGr_{(a-2)}\oplus \EGr_2$ is a trivial bundle, we have $P(\EGr_{(a-2)})P(\EGr_2)=1$. From Lemma~\ref{lem:rank2pi1} we know $P(\EGr_2)=1+[\frac{1}{2\pi}\Omega_{12}]^2$, and so $P(\EGr_{(a-2)})=\sum_i(-1)^i[\frac{1}{2\pi}\Omega_{12}]^{2i}$. Thus $p_{\lfloor\frac{a-2}{2}\rfloor}(E_{(a-2)})=(-1)^{\lfloor\frac{a-2}{2}\rfloor}[\frac{1}{2\pi}\Omega_{12}]^{2{\lfloor\frac{a-2}{2}\rfloor}}=(-1)^{\lfloor\frac{a-2}{2}\rfloor}[\frac{1}{2\pi}\Omega_{12}]^{a-3}$.

It remains to show that 
$[\frac{1}{2\pi}\Omega_{12}]^{a-2}\frown[\Gr(2,\R^a)]=1$. (Here the fundamental class $[\Gr(2,\R^a)]$ is in $H_*(M,\DD)$.) Since $[\frac{1}{2\pi}\Omega_{12}]$ is the twisted Euler class of $\EGr_2$, the class $[\frac{1}{2\pi}\Omega_{12}]^{a-2}$ is the twisted Euler class of $\bigoplus^{a-2} \EGr_2$. As in the proof of Lemma~\ref{lem:p1}, we can find the Poincare dual of this class by counting zeroes of a section. Define $s:\Gr(2,\R^a)\to \bigoplus^{a-2} \EGr_2$ by $s(V)=(V,(\pi_V(\e_1),\ldots,\pi_V(\e_{a-2})))$, where $\pi_V$ is orthogonal projection onto $V$. This is a section with a single zero at $V_0:=\langle \e_{a-1}, \e_a\rangle$. A calculation similar to that in the proof of Lemma~\ref{lem:p1} shows this zero to have degree 1.
\end{proof}

Thus the analog to Lemma~\ref{lem:pistar} is the following.
\begin{lemma} For all $a$ we have $\pi^!(p_{\lfloor\frac{a-2}{2}\rfloor}(E_{(a-2)})[\frac{1}{2\pi}\Omega_{12}])=(-1)^{\lfloor\frac{a-2}{2}\rfloor}$.
\end{lemma}

The proof of Theorem~\ref{thm:chernweil} for odd $a$ then follows the same lines as the proof for even $a$, by showing that $P(E)\smile\pi^!(\sum_i(-1)^i\left[\frac{1}{2\pi}\Omega_{12}\right]^{2i+1})=(-1)^{\lfloor\frac{a-2}{2}\rfloor}$.

\section{CD manifolds and oriented matroid flags}

Throughout the following we refer to ``the oriented matroid associated to a vector arrangement" and ``the oriented matroid associated to a subspace". The reader unfamiliar with these should first read Appendix~\ref{app:OM}.

\subsection{Combinatorial  charts and CD manifolds}

The bridge from smooth to combinatorial thinking comes via {\em combinatorial differential manifolds}, or {\em CD manifolds}, a concept that is implicit in~\cite{GM} and was made explicit in the expository paper~\cite{MacP}.

Throughout the following, a {\em simplicial complex} is an abstract simplicial complex, i.e., a collection of finite sets closed under taking subsets. If $X$ is a simplicial complex, $\|X\|$  denotes the geometric realization of $X$ in $\R^{X^0}$, and for each $\Delta\in X$, $\|\Delta\|$ denotes the corresponding closed simplex. Thus an element of $\|\Delta\|$ is a sum $\sum_{e\in\Delta} a_e e$, where each $a_e$ is a nonnegative real number and $\sum_e a_e=1$. By the {\em star} of a simplex $\Delta$ of $X$, we mean the closed star (as a simplicial complex). For each $p\in\|X\|$ we let $\Delta_p$ denote the smallest element of $X$ such that $p\in\|\Delta_p\|$. 


\begin{defn} Let $X$ be a simplicial manifold. A {\em smoothing} of $X$, or {\em smooth structure} on $X$, is a smooth manifold $M$ and a homeomorphism $\eta:\|X\|\to M$ such that, for each simplex $\sigma$ in $X$, the restriction of $\eta$ to $\|\sigma\|$ extends to a smooth embedding of an open neighborhood of $\|\sigma\|$ in its linear span.
\end{defn}

\begin{defn}\cite{MacP} Let $X$ be a simplicial manifold of dimension $n$, and let $\Delta$ be a simplex of $X$. A {\em flattening} of $X$ at $\Delta$ is a simplex-wise linear homeomorphism $f:\|\Star\Delta\|\to U\subset V$ of the star of $\Delta$ onto a neighborhood $U$ of the origin in an $n$-dimensional real vector space $V$ such that the image of the interior of $\|\Delta\|$ contains the origin.
\end{defn}

Let $\eta:\|X\|\to M$ be a smoothing. For each $p$ in $\|X\|$ the smoothing 
induces a flattening $f_p:\|\Star \Delta_p\|\to T_{\eta(p)}M$ of $X$ as follows. For each vertex $v$ of $\Star(\Delta_p)$, let $\lambda:[0,1]\to\|\Delta\cup\{v\}\|$ be the linear path from $p$ to $v$. Then we define $f_p(v)=(\eta\circ\lambda)'(0)$ and define $f_p$ to be linear on simplices.

We think of this flattening as a piecewise-linear  analog to a coordinate chart at $p$. The idea of CD manifolds is to go one step further and replace each flattening with an oriented matroid derived from that flattening.

There are two ways one can derive an oriented matroid from a flattening $f:\|\Star(\Delta)\|\to V$. The first is simply to take the rank $n$ oriented matroid of the vector arrangement $(f(v):v\in\Star(\Delta)^0)$. This is the motivation for the {\em linear OM coordinate charts} described below. The second way is to 
take the rank $n+1$ oriented matroid of the vector arrangement $((f(v),1):v\in\Star(\Delta)^0)$ in $V\oplus \R$. This is the motivation for the {\em affine OM coordinate charts} described below. Gelfand and MacPherson used affine OM coordinate charts in~\cite{GM}, while MacPherson used linear OM coordinate charts in~\cite{MacP}.  Neither approach is stronger or weaker than the other, and a combinatorial formula for the Pontrjagin classes can be derived from either. We'll present both.

\begin{defn} Let $X$ be a simplicial manifold of dimension $n$, and let $\Delta$ be a simplex of $X$. An {\em affine OM chart} at $\Delta$ is a rank $n+1$ oriented matroid $\M$ on elements $X^0$ satisfying all of the following.
	\begin{enumerate}
		\item The nonloop\footnote{We use the term {\em loop} that is standard in the matroid literature rather than the admittedly more intuitive term {\em zero element} used in~\cite{GM}.}
 elements of $\M$ are exactly the vertices of $\mathrm{star}(\Delta)$.
		\item  The sign vector with value $+$ on all vertices of $\mathrm{star}(\Delta)$ and 0 on all other elements is a covector of $\M$.
		\item If $\sigma$ is a simplex in $\mathrm{star}(\Delta)$ then $\sigma$ is independent in $\M$.
		\item For each  pair $\sigma, \tau$ of simplices in $\mathrm{star}(\Delta)$ there is a covector $Y$ of $\M$ such that  $Y(\sigma)\subseteq\{0,+\}$ and $Y(\tau)\subseteq\{0,-\}$.
	\end{enumerate}
\end{defn}

\begin{defn}\cite{MacP} Let $X$ be a simplicial manifold of dimension $n$, and let $\Delta$ be a simplex in $X$. A {\em linear OM  chart} at $\Delta$ is a rank $n$ oriented matroid $\M$ on elements $X^0$  satisfying all of the following.
	\begin{enumerate}
		\item All elements of $X^0-\mathrm{star}(\Delta)^0$ are loops in $\M$.
			\item $\Delta$ is dependent in $\M$.
		\item If $\sigma$ is a simplex in the boundary of $\mathrm{star}(\Delta)$ then $\sigma$ is independent in $\M$, and no other nonloop element is in the convex hull of $\sigma$.
	\end{enumerate}
\end{defn}

Given a flattening of $\Delta$ at $X$, the corresponding linear OM  chart is a strong map image of the corresponding affine OM  chart. An affine OM  chart may have several linear OM  charts as strong map images: see Figure~\ref{fig:afftolin} for a rank 3 affine OM chart with three linear OM charts as strong map images. Also, a linear OM chart might be a strong map image of several affine OM charts: see Figure~\ref{fig:lintoaff} for a rank 2 linear OM chart which is a strong map image of three different affine OM charts.

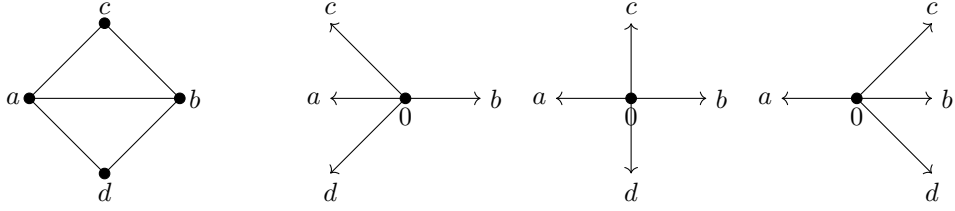
\begin{figure}
	\begin{tikzpicture}
	\draw (-1,0)--(1,0);
	\draw(-1,0)--(0,1);
	\draw(0,1)--(1,0);
	\draw(-1,0)--(0,-1);
	\draw(0,-1)--(1,0);
	\filldraw[black] (-1,0) circle (2pt) node[anchor=east] {$a$};
	\filldraw[black] (1,0) circle (2pt) node[anchor=west] {$b$};
	\filldraw[black] (0,1) circle (2pt) node[anchor=south] {$c$};
	\filldraw[black] (0,-1) circle (2pt) node[anchor=north] {$d$};
	
	\draw[<->]  (3,0)node[anchor=east] {$a$}--(5,0) node[anchor=west] {$b$} ;
	\draw[->] (4,0)--(3,1) node[anchor=south] {$c$};
	\draw[->] (4,0)--(3,-1) node[anchor=north] {$d$};
	\filldraw[black] (4,0) circle (2pt) node[anchor=north] {$0$};
	
	\draw[<->]  (6,0)node[anchor=east] {$a$}--(8,0) node[anchor=west] {$b$};
	\draw[->] (7,0)--(7,1) node[anchor=south] {$c$};
	\draw[->] (7,0)--(7,-1) node[anchor=north] {$d$};
	\filldraw[black] (7,0) circle (2pt) node[anchor=north] {$0$};
	
	\draw[<->]  (9,0)node[anchor=east] {$a$}--(11,0) node[anchor=west] {$b$};
	\draw[->] (10,0)--(11,1) node[anchor=south] {$c$};
	\draw[->] (10,0)--(11,-1) node[anchor=north] {$d$};
	\filldraw[black] (10,0) circle (2pt) node[anchor=north] {$0$};
	\end{tikzpicture}
	\caption{A rank 3 affine OM  chart and three rank 2 linear OM charts\label{fig:afftolin}}
\end{figure}

\begin{figure}
	\begin{tikzpicture}
	\draw[<->]  (-1,0)node[anchor=east] {$a$}--(1,0) node[anchor=west] {$b$} ;
	\draw[->] (0,0)--(-1,1) node[anchor=south] {$c$};
	\draw[->] (0,0)--(-1,-1) node[anchor=north] {$d$};
	\filldraw[black] (0,0) circle (2pt) node[anchor=north] {$0$};
	
	\draw (3,0)--(5,0);
	\draw(3,0)--(2.5,1);
	\draw(2.5,1)--(5,0);
	\draw(3,0)--(2.5,-1);
	\draw(2.5,-1)--(5,0);
	\filldraw[black] (3,0) circle (2pt) node[anchor=east] {$a$};
	\filldraw[black] (5,0) circle (2pt) node[anchor=west] {$b$};
	\filldraw[black] (2.5,1) circle (2pt) node[anchor=south] {$c$};
	\filldraw[black] (2.5,-1) circle (2pt) node[anchor=north] {$d$};
	
	\draw (6,0)--(8,0);
	\draw(6,0)--(6,1);
	\draw(6,1)--(8,0);
	\draw(6,0)--(6,-1);
	\draw(6,-1)--(8,0);
	\filldraw[black] (6,0) circle (2pt) node[anchor=east] {$a$};
	\filldraw[black] (8,0) circle (2pt) node[anchor=west] {$b$};
	\filldraw[black] (6,1) circle (2pt) node[anchor=south] {$c$};
	\filldraw[black] (6,-1) circle (2pt) node[anchor=north] {$d$};
	
	\draw (9,0)--(11,0);
	\draw(9,0)--(10,1);
	\draw(10,1)--(11,0);
	\draw(9,0)--(10,-1);
	\draw(10,-1)--(11,0);
	\filldraw[black] (9,0) circle (2pt) node[anchor=east] {$a$};
	\filldraw[black] (11,0) circle (2pt) node[anchor=west] {$b$};
	\filldraw[black] (10,1) circle (2pt) node[anchor=south] {$c$};
	\filldraw[black] (10,-1) circle (2pt) node[anchor=north] {$d$};
	\end{tikzpicture}
	\caption{A rank 2 linear OM chart and three rank 3 affine OM charts\label{fig:lintoaff}}
\end{figure}
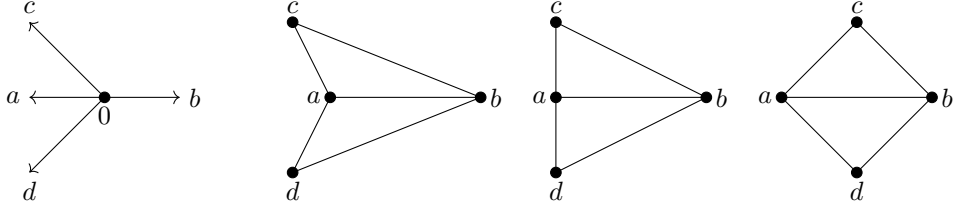

\begin{example}\label{exam:torus1} Let $M$ be the torus $\R^2/\mathbb Z^2$ and $\eta: \|X\|\to M$ be the ``usual" triangulation, the image under the quotient map of the triangulation of $[0,1]^2$ given in Figure~\ref{fig:torus1}. 

\begin{figure}
\begin{tikzpicture}
\draw (0,0)--(3,0)--(3,3)--(0,3)--(0,0) node[below] at (0,0){$(0,0)$} node[below] at (3,0){$(1,0)$} node[above] at (0,3){$(0,1)$} node[above] at (3,3){$(1,1)$};
\draw (0,1)--(3,1);
\draw (0,2)--(3,2);
\draw (1,0)--(1,3);
\draw (2,0)--(2,3);
\draw (0,2)--(1,3);
\draw (0,1)--(2,3);
\draw (0,0)--(3,3);
\draw (1,0)--(3,2);
\draw(2,0)--(3,1);
\node[right] at (1,1.85){$a$} node[above] at (1,3){$b$} node[left] at (0,2){$c$} node [left] at (0,1){$d$} node[left] at (1,1.15){$e$} node[left] at (2,2.15){$f$} node[above] at (2,3){$g$};
\end{tikzpicture}
\caption{A triangulation of the torus, with names for some vertices \label{fig:torus1}}
\end{figure}
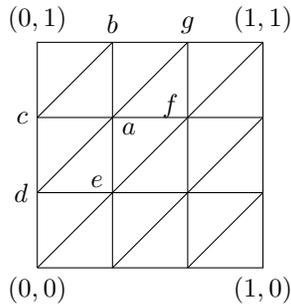

The vector arrangement arising from the flattening at $a$ is shown in Figure~\ref{fig:torus2}. (Note that $a$ itself corresponds to the vector $\mathbf 0$.)
\begin{figure}
\begin{tikzpicture}
\draw[->] (0,0)--(1,0) node[right] at (1,0){$f$};
\draw[->] (0,0)--(1,1) node[right] at (1,1){$g$};
\draw[->] (0,0)--(0,1) node[above] at (0,1){$b$};
\draw[->] (0,0)--(-1,0) node[left] at (-1,0){$c$};
\draw[->] (0,0)--(-1,-1) node[left] at (-1,-1){$d$};
\draw[->] (0,0)--(0,-1) node[below] at (0,-1){$e$};
\node at (-.1,.1){$a$};
\end{tikzpicture}
\caption{The vector arrangement from the flattening at $a$  \label{fig:torus2}}
\end{figure}
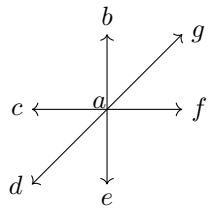

The  linear OM chart at $a$ is the unique rank 2 oriented matroid on elements $X^0$ with nonloop elements $b,c,d,e,f,g$ in which $b$ and $e$ are antiparallel, $c$ and $f$ are antiparallel,  $d$ and $g$ are antiparallel, and in any realization either the clockwise order or counterclockwise order of the nonzero elements is $b,c,d,e,f,g$.   The affine OM chart at $a$ is the unique acyclic rank 3 oriented matroid on elements $X^0$  with nonloop elements $a,b,c,d,e,f,g$ in which $a$ is in the convex hull of each of  $\{b,e\}$, $\{c,f\}$, and $\{d,g\}$ and in any realization by an affine point configuration the elements $b,c,d,e,f,g$ are the vertices of a convex hexagon, with the indicated cyclic order.

Now consider the vector arrangements associated to points on the edge $ab$. Figure~\ref{fig:torus3} shows the arrangements at 3 points along this edge: first, at a point close to vertex $a$, second, at the midpoint of the edge, and third, near the point $b$. 

\begin{figure}
\begin{tikzpicture}
\draw[->] (0,0)--(1,.8) node[right] at (1,.8){$g$};
\draw[->] (0,0)--(0,.8) node[above] at (0,.8){$b$};
\draw[->] (0,0)--(-1,-.2) node[left] at (-1,-.2){$c$};
\draw[->] (0,0)--(0,-.2) node[below] at (0,-.2){$a$};

\begin{scope}[xshift=1.5in, yshift=.1in]
\begin{scope}[scale=.5]
\draw[->] (0,0)--(0,1) node[above] at (0,1){$b$};
\draw[->] (0,0)--(0,-1) node[below] at (0,-1){$a$};
\draw[->] (0,0)--(2,1) node[right] at (2,1){$g$};
\draw[->] (0,0)--(-2,-1) node[left] at (-2,-1){$c$};
\end{scope}
\end{scope}

\begin{scope}[xshift=3in, yshift=.1in]
\draw[->] (0,0)--(1,.2) node[right] at (1,.2){$g$};
\draw[->] (0,0)--(0,.2) node[above] at (0,.2){$b$};
\draw[->] (0,0)--(-1,-.8) node[left] at (-1,-.8){$c$};
\draw[->] (0,0)--(0,-.8) node[below] at (0,-.8){$a$};
\end{scope}
\end{tikzpicture}
\caption{Vector arrangements from flattenings at three points on $ab$  \label{fig:torus3}}
\end{figure}
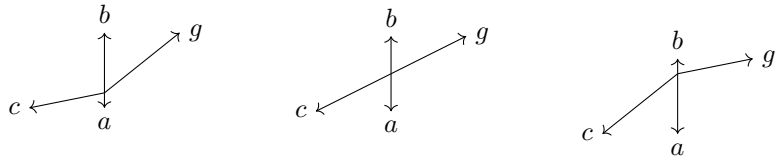

These three vector arrangements realize the three different linear OM charts that arise from flattenings at points on $ab$. (To see the difference in the oriented matroids, note that the convex hull of $\{c,g\}$ is $\{b,c,g\}$ in the first chart, is $\{c,g\}$ in the  the second chart, and is $\{a,c,g\}$ in the third chart.) All three of these arrangements give the same affine OM chart, namely the unique rank 3 oriented matroid arising from a convex quadrilateral $agbc$ in an affine plane.

Some things to notice:
\begin{enumerate}
\item All points in the interior of a maximal simplex have the same linear OM chart and the same affine OM chart. (This is true for any smoothing of any simplicial manifold.)
\item In this example, the affine OM chart is constant on the interior of each simplex. There is a regular cell decomposition $\hat X$ of $\|X\|$ so that the linear OM chart is constant on the interior of each cell. (The subdivision just divides each 1-simplex in half.) 

In general, we say that a smoothing of a simplicial complex $Y$  is {\em tame} (with respect to linear resp.\  affine charts) if there is a regular cell decomposition $\hat Y$ of $\|Y\|$ such that the resulting linear resp.\  affine OM chart is constant on the interior of each cell. If  we have such a regular cell decomposition, then we denote the chart associated to a cell $\sigma$ by $\Phi(\sigma)$.

\item For cells $\sigma'\subset\sigma$ of $\hat X$,
$\Phi(\sigma')$ is obtained from $\Phi(\sigma)$ by moving elements into more special position and then perhaps replacing some loops with nonloops. (This is true for any smoothing of any simplicial manifold.)
\end{enumerate}
\end{example}

A triple $(X,\hat X,\Phi)$ as in the previous example is an example of a {\em combinatorial differential manifold}.

%

When $\hat X$ is a subdivision of $\|X\|$ and $\sigma\in\hat X$, we  let $\Delta_\sigma$ denote the smallest simplex of $X$ with $\sigma\subseteq\|\Delta\|$.

\begin{defn} A {\em combinatorial differential (CD) manifold (with an affine resp.\  linear  atlas)} is a simplicial manifold $X$ together with a regular cell complex $\hat{X}$ which is a subdivision  of $\|X\|$ and an assignment of a linear resp.\  affine OM chart $\Phi(\sigma)$ at $\Delta_\sigma$ to each cell $\sigma$ of $\hat{X}$, 
such that, for each pair $\tau\subset\sigma$ of cells of $\hat{X}$, the restriction of $\Phi(\tau)$ to the elements of $\Phi(\sigma)$ is a weak map image of $\Phi(\sigma)$.
\end{defn}

\begin{remark} This is not as general as the definition of CD manifold given in ~\cite{MacP} and~\cite{mythesis}. Both of those sources  pursue a purely combinatorial theory. They remove the restriction that $X$ be a simplicial manifold and instead add combinatorial conditions that aspire to be a full characterization of the combinatorics of coordinate charts. In~\cite{MacP} it was conjectured that these conditions do, in fact, imply that $X$ is a simplicial manifold. This conjecture was proved in~\cite{mythesis} for CD manifolds in which all charts are Euclidean.
\end{remark}

\subsection{Flattenings and classifying maps}
 
 The classifying space for rank $n$ real vector bundles is the Grassmannian $\Gr(n,\R^\infty)$. Thus, a smooth $n$-dimensional manifold $M$ has a {\em classifying map} $M\to \Gr(n,\R^\infty)$ such that the pullback of the canonical bundle over $\Gr(n,\R^\infty)$ is isomorphic to $TM$. In this section we'll describe how a smoothing $\eta: \|X\|\to M$ gives such a  map, specifically from $M$ to $\Gr(n,\R^{X^0})$.

As in Appendix~\ref{app:OM},  it will be convenient to view a matrix as an arrangement of column vectors and to allow the columns to be indexed by an arbitrary finite set $S$, not necessarily $[m]$. Thus we'll feel free to view $(\vv_i:i\in S)$ as either an arrangement of vectors or as an object with, for instance, a row space (contained in $\R^S$).
  
 The flattenings induced by a smooth triangulation of $M$ give such vector arrangements: as we have seen, at each $p\in M$ we get a vector arrangement $(f_p(v):v\in\Star(\Delta_p)^0)$, and by setting $f_p(v)=0$ for each $v\in {X^0}-\Star(\Delta_p)^0$ we get a vector arrangement $(f_p(v):v\in {X^0})$.  But the resulting map $M\to \Gr(n,\R^{X^0})$ taking each $p$ to $\row(f_p(v):v\in {X^0})$
  is not continuous: typically, as we go from the interior of some $\eta(\Delta)$ to its boundary, the number of nonzero vectors in our arrangement increases. 
 
 To get a continuous map $M\to \Gr(n,\R^{X^0})$, we'll extend our vector arrangements a bit.
 
 \begin{defn} Let  $i\in {X^0}$, and let $\epsilon>0$. The {\em $\epsilon$-star} $\|\epsilon\Star(i)\|$ of $i$  is the subset of $\R^{X^0}$ consisting of points  of the form $\sum_{f\in\tau} a_f f$, where $\tau\in\Star(i)$, $\sum_{f\in\tau} a_f =1$, 
 $a_f\geq 0$ for all $f\neq i$, and $a_i>-\epsilon$.
 \end{defn}
 
 Because $\eta:\|X\|\to M$ is smooth on closed simplices, the restriction of $\eta$ to a single star $\|\Star(i)\|$ extends to a smooth map $\eta_i$ from some $\epsilon$-star of $i$ to $M$. Fix such a map $\eta_i$ for each $i\in {X^0}$, and let $U_i$ be the image of $\eta_i$.  Also, fix a smooth map $b_i:M\to[0,1]$ that has value 1 on $\eta(\|\Star(i)\|)$ and value 0 on the complement of $U_i$. For each $p\in M$ we define a vector arrangement $(\vv_{i,p}:i\in {X^0})$ in $T_p(M)$ as follows. 
 \begin{itemize} 
 \item If $p\in U_i$ then let $\lambda_i$ be the linear path from $\eta_i^{-1}(p)$ to $i$ in $\|\epsilon\Star(i)\|$, and let $\vv_{i,p}=b_i(p)(\eta_i\circ\lambda_i)'(0)$.
 \item If $p\not\in U_i$ then let $\vv_{i,p}=0$.
 \end{itemize}

 \begin{example} Looking again at Example~\ref{exam:torus1}, let's consider $U_e$. Fix  a homeomorphism $\eta_e$ from an $\epsilon$-star of $e$ to the shaded area in the torus shown in Figure~\ref{fig:epsilonstar} that coincides with $\eta$ on $\|\Star(e)\|$ and  is linear with respect to each simplex in the torus.
 	\begin{figure}
 		\begin{tikzpicture}[fill opacity=0.5]
		\fill[gray] (1.2,0)--(2.2,.8)--(2.2,2.2)--(.8,2.2)--(0,1.2)--(0,0)--(1.2,0);
		\fill[gray] (1.2,3)--(1.2,2.8)--(0,2.8)--(0,3)--(1.2,3);
		\fill[gray]  (3,1.2)--(2.8,1.2)--(2.8,0)--(3,0)--(3,1.2);
		\fill[gray]  (2.8,3)--(2.8,2.8)--(3,2.8)--(3,3)--(2.8,3);
 		\draw (0,0)--(3,0)--(3,3)--(0,3)--(0,0) ;
 		\draw (0,1)--(3,1);
 		\draw (0,2)--(3,2);
 		\draw (1,0)--(1,3);
 		\draw (2,0)--(2,3);
 		\draw (0,2)--(1,3);
 		\draw (0,1)--(2,3);
 		\draw (0,0)--(3,3);
 		\draw (1,0)--(3,2);
 		\draw(2,0)--(3,1);
 		\node[right] at (1,1.85){$a$} node[above] at (1,3){$b$} node[left] at (0,2){$c$} node [left] at (0,1){$d$} node[left] at (1,1.15){$e$} node[left] at (2,2.15){$f$} node[above] at (2,3){$g$};
		\draw[dashed] (1.2,0)--(2.2,.8)--(2.2,2.2)--(.8,2.2)--(0,1.2);
		\draw[dashed] (1.2,3)--(1.2,2.8)--(0,2.8);
		\draw[dashed] (3,1.2)--(2.8,1.2)--(2.8,0);
		\draw[dashed] (2.8,3)--(2.8,2.8)--(3,2.8);
 		\end{tikzpicture}
 		\caption{The set $U_e$\label{fig:epsilonstar}}
 	\end{figure}
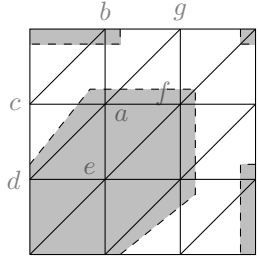
	Consider a point $p$ on $\eta(\|\{a,c\}\|)$ that lies in $U_e$ and satisfies $b_e(p)\neq 0$. The line segment from $\eta_e^{-1}(p)$ to $e$ in $\|\epsilon\Star(e)\|$ is mapped by $\eta_e$ to the union of two line segments, one in $\eta(\|\{a,c,d\}\|)$ and one in $\eta(\|\{a,d,e\}\|)$. A tangent at $p$ to the line segment in $\eta(\|\{a,c,d\}\|)$  gives the vector $\vv_{e,p}$.
	
	In a similar way,  each point $p$ that lies in the interior of  $\eta(\|\{a,c\}\|)$ and is sufficiently close to $a$ has associated vector arrangement $(\vv_{i,p}:i\in {X^0})$ whose nonloop elements are $\{a,b,c,d,e,f,g\}$.
 \end{example}
 
   Let $A_\ll:M\to \Gr(n,\R^{X^0})$ be the map sending each $p$ to $\row(\vv_{i,p}:i\in {X^0})$.

 \begin{prop} \label{prop:pullbacklin} $TM$ is isomorphic to the pullback of the canonical bundle by $A_\ll$.
 \end{prop}
 
 \begin{proof} For each $p\in M$,  $A_\ll(p)$ 
  is the vector space of all vectors $(\y\cdot\vv_{e,p}: e\in X^0)$ with $\y\in T_pM$. An isomorphism $i$ from $TM$ to the pullback of the canonical bundle over $\Gr(n,\R^{X^0})$ is given by $i(p,\y)=(p,\row(\y\cdot\vv_{e,p}: e\in X^0))$. This map is continuous and linear on fibers, and it's an injection (and therefore a bijection) because the matrix $(\vv_{e,p}:e\in {X^0})$ is full rank.
 \end{proof}

 \begin{defn} $\Gr_\aa(n+1,\R^S)$ is the set of all $V\in \Gr(n+1,\R^S)$ such that $(1,1,\ldots,1)\in V$.
  \end{defn} 
  
Once again, from a triangulation of a smooth manifold $M$ we get a vector arrangement in $T_pM$ for each $p\in M$. By adding a last component of 1 to each vector we get a vector arrangement in $T_pM\oplus\R$. As in the linear case we can use $\epsilon$-stars to get a continuous map $A_\aa:M\to\Gr_\aa(n+1,\R^{X^0})$.
By much the same proof as for Proposition~\ref{prop:pullbacklin}, we get the following.

 \begin{prop}\label{prop:pullbackaffine} $TM\oplus\mathbf 1$ is isomorphic to the pullback of the canonical bundle by $A_\aa$.
 \end{prop}

\subsection{The MacPhersonian\label{sec:MacP}}

Throughout the following, the order complex of a poset $P$ (that is, the simplicial complex of all totally ordered subsets of $P$) is denoted $\Cx P$. In particular, if $P$ is the poset of cells in a regular cell complex then $\Cx P$ is its barycentric subdivision. If $f:P\to Q$ is a poset map then $\Cx f:\Cx P\to\Cx Q$ is the induced map of abstract simplicial complexes. To streamline notation we will denote $\|\Cx P\|$ as $\|P\|$ and denote the map $\|\Cx f\|$ from $\|P\|$ to $\|Q\|$ by $\|f\|$.

As described in Appendix~\ref{app:OM}, a strong map image of an oriented matroid is analogous to a subspace of a vector space. Thus the set of all rank $n$ strong map images of a fixed rank $m$ oriented matroid should be analogous to $\Gr(n,\R^m)$.  This analogy can be pursued in various ways (see Section~\ref{sec:eric}). The particular such set of interest to us is the set of strong map images of the {\em coordinate oriented matroid}, the unique rank $|S|$ oriented matroid with elements $S$, for $S$ a finite set. The coordinate oriented matroid on elements $S$  is realized by the  coordinate vectors in $\R^S$. Every oriented matroid  is a strong map image of the coordinate oriented matroid on the same set of elements. Thus the ``oriented matroid Grassmannians" of interest to us are given by the following definition.

\begin{defn} The {\em MacPhersonian} $\MacP(n,S)$ is the poset of all rank $n$ oriented matroids with elements $S$, ordered by weak maps.
  
  $\MacP_\aa(n+1, S)$ is the set of all $\M\in\MacP(n+1,S)$ such that the sign vector with all coordinates $+$ is a covector of $M$.

\end{defn}

Using the correspondence between subspaces of $\R^S$ and realizable oriented matroids described in Appendix~\ref{app:OM},  we get  maps $\mu:\Gr(n, \R^S)\to\MacP(n,S)$ and  $\mu_\aa: \Gr_\aa(n+1,\R^S)\to\MacP_\aa(n+1, S)$. 

\begin{prop} \label{prop:muproperties} The maps $\mu$ and $\mu_\aa$ are semialgebraic and upper semicontinuous. That is:
\begin{enumerate}
\item The preimage of each element of $\MacP(n,S)$ resp.\  $\MacP_\aa(n+1, S)$ is a semialgebraic set.
\item For each $V\in \Gr(n, \R^S)$ there is an open neighborhood $U$ of $V$ in $\Gr(n, \R^S)$ such that $\mu(W)\leadsto\mu(V)$ for each $W\in U$.
\end{enumerate}
\end{prop}

Another way to state the second part of the proposition is to say that $\mu$ and $\mu_\aa$ are continuous with respect to the topology on $\MacP(n,S)$ resp.\  $\MacP_\aa(n+1,S)$ in which the open sets are the filters.
\begin{proof} For convenience we assume $S=\{1,\ldots,m\}$. Both parts follow from the chirotope characterization of oriented matroids: for each $V=\row(M)\in \Gr(n,\R^m)$, $\mu(V)$ is the oriented matroid with chirotope $\chi(i_1, \ldots, i_n)=\sign|M_{i_1,\ldots,i_n}|$.

1. Let $\M\in\MacP(n,[m])$, and without loss of generality assume that $\{1,\ldots,n\}$ is a basis for $\M$. Let $\chi$ be the chirotope for $\M$ with $\chi(1,\ldots, n)=+$. Then $\mu^{-1}(\M)$ is isomorphic to the set of all matrices $(I|A)$ satisfying the equalities and inequalities on determinants given by $\chi$.

2. Without loss of generality  assume that $V=\row(I|A)$ for some $A$. Let $S$ be the subset of $\Gr(n, \R^m)$ of matrices of the form $\row(I|\hat A)$. Thus we  have a coordinate chart identifying each element  $\row(I|\hat A)$ of $S$ with $\hat A$. Also, for each $W=\row(I|\hat A)\in S$ we can identify $\mu(W)$ with the chirotope $\chi$ for $\mu_W$ for which $\chi(i_1, \ldots, i_n)=\sign|(I|\hat A)_{i_1, \ldots, i_n}|$ for each $1\leq i_1<\ldots<i_n\leq m$. The result then follows from continuity of the determinant function.
\end{proof}

 It follows (Theorem~\ref{whitney:semialgebraic}) that the set of preimages of elements of $\MacP(n,S)$ resp.\  $\MacP_\aa(n+1, S)$  can be subdivided to a Whitney stratification. Further, for every such stratification and every pair $\sigma, \tau$ of strata with $\tau\subset\overline{\sigma}$ the oriented matroid associated to $\sigma$ weak maps to the oriented matroid associated to $\tau$.

\subsection{General combinatorial Grassmannians\label{sec:eric} }

\begin{defn} Let $\M$ be an oriented matroid, and let $r$ be a natural number. The {\em OM Grassmannian} $\Gr(r, \M)$ is the poset of all rank $r$ strong map images of $\M$, ordered by weak maps. 
\end{defn}

Consider an oriented matroid $\M$  realized by a vector arrangement $\A$ spanning a space $W$ equipped with an inner product. For each rank $r$ subspace $V$ of $W$, let $\pi_V$ denote orthogonal projection onto $V$. The arrangement $(\pi_V(\vv):\vv\in\A)$ is a vector realization of an element of $\Gr(r, \M)$.  Thus $\A$ determines a map
$\Gr(r,W)\to\Gr(r, \M)$.
In Section~\ref{sec:flags} we will consider this map when $W\subseteq\R^n$ and  $\A$ is the orthogonal projection of the coordinate vector arrangement onto $W$. In this situation the map $\Gr(r,W)\to\Gr(r, \M)$ depends only on $W$, so without ambiguity we can denote it by $\mu$. As we have already seen with the MacPhersonian, there is a Whitney stratification of $\Gr(r,W)$ subdividing the partition into preimages under $\mu$. Subdividing this stratification further to a triangulation gives a simplicial map $\tilde\mu:\Gr(r,W)\to\|\Gr(r, \M)\|$, well-defined up to homotopy

It follows from the Topological Representation Theorem that  $\|\Gr(1,\M)\|$ is homeomorphic to $\Gr( 1,\R^{\rank(\M)})$ for every $\M$.
 In~\cite{Bab} Babson showed that for every $\M$,  $\|\Gr(2, \M)\|$ has the homotopy type of $\Gr(2, \R^{\rank(\M)})$, and each map $\tilde\mu: \Gr(2,W)\to\|\Gr(2,\M)\|$ induced by a realization of $\M$  is a homotopy equivalence. Recently the first author showed that $\|\Gr(2,\M)\|$ is homeomorphic to $\Gr(2, \R^{\rank(\M)})$ (\cite{Abawrank2}) for all $\M$.
 
 In contrast, there are rank 4 oriented matroids $\M$, even with $\M$ realizable, such that $\|\Gr(3,\M)\|$ is disconnected (\cite{MRG}, \cite{Wu2021}).

\subsection{Tame classifying maps\label{sec:tameclass}}

\begin{defn} Let $X$ be a simplicial complex. A map $f:\|X\|\to \Gr(n,\R^{X^0})$ is {\em tame} if there is a regular cell decomposition of $\|X\|$ subdividing the simplicial decomposition such that the composition $\mu\circ f$ is constant on the interior of each cell.
\end{defn}


In this section we show that if  $\eta:\|X\|\to M$ is a triangulation and $A_\aa:M\to \Gr_\aa(n+1, \R^{X^0})$ is the resulting affine classifying map then there is a map $\hat A_\aa$ arbitrarily close to $A_\aa$ such that $\hat A_\aa\circ\eta$ is tame, and that such an $\hat A_\aa$, or a tame $A_\ll$, leads to an affine resp.\ linear CD manifold structure on $X$.

The map $A_\aa$ is continuous and smooth on $\eta(\|\Delta\|)$ for each $\Delta\in X$. Since $M$ is compact, $X$ is finite, and so there is a smooth map $\hat A_\aa:M\to \Gr(n, \R^{X^0})$  arbitrarily close to $A_\aa$. Recall from Section~\ref{sec:MacP} that the oriented matroid partition of $\Gr_\aa(n, \R^{X^0})$ can be subdivided to a Whitney stratification. Also, the smooth triangulation of $M$ gives a Whitney stratification of $M$ in which the strata are the images of simplices. By Theorem~\ref{whitney:opendense} we may assume that each restriction of $\hat A_\aa$ to the image of a simplex is transverse to this Whitney stratification. Thus by Theorems~\ref{whitney:pullback} and~\ref{whitney:triangulate} $\mu\circ\hat A_\aa\circ\eta$ is tame.

\begin{prop}\label{prop:tame}
 \begin{enumerate}
\item Let $\hat A_\aa$ be as above. Let $\hat X$ be a regular cell decomposition of $\|X\|$ subdividing the simplicial decomposition such that the composition $\mu\circ \hat A_\aa\circ\eta$ is constant on the interior of each cell. For each $\sigma\in\hat X$, let $\Phi(\sigma)$ be obtained from $\mu\circ\hat A_\aa\circ\eta (\sigma)$ by making each element of $X^0-\Star(\Delta_\sigma)^0$ a loop.

 Assuming $\hat A_\aa$ is taken sufficiently close to $A_\aa$,  $(X,\hat X,\Phi)$ is a CD manifold with an affine coordinate chart.
 \item Let $A_\ll:\Gr(n,\R^{X^0})$ be a linear classifying map arising from a smoothing $\eta$ of $X$ such that $A_\ll\circ \eta$ is tame. Let $\hat X$ be a regular cell decomposition of $\|X\|$ subdividing the simplicial decomposition such that the composition $\mu\circ A_\ll\circ\eta$ is constant on the interior of each cell. For each $\sigma\in\hat X$, let $\Phi(\sigma)$ be obtained from $\mu\circ\hat A_\aa\circ\eta (\sigma)$ by making each element of $X^0-\Star(\Delta_\sigma)^0$ a loop. Then $(X,\hat X,\Phi)$ is a CD manifold with a linear coordinate chart.
\end{enumerate}

\end{prop}

\begin{remark} A linear coordinate chart at a simplex $\Delta$ must satisfy the non-generic condition that $\Delta$ is dependent. Thus a small perturbation $\hat A_\ll$ of $A_\ll$ need not give linear coordinate charts.
\end{remark}

\begin{proof} 1. Let $\sigma\in\hat X$. By Proposition~\ref{prop:muproperties}.2, if $\hat A_\aa$ is taken sufficiently close to $A_\aa$ then for each $p\in\sigma$ we have $\Phi(\sigma)=\mu\circ\hat A_\aa(p)\leadsto\mu\circ A_\aa(p)$. It is easy to see that if $\M$ is an affine coordinate chart at $\Delta$, $\hat\M\leadsto\M$, and $\hat\M$ has the same set of loops as $\M$  then $\hat\M$ is an affine coordinate chart at $\Delta$. 

Also, from upper semi-continuity (Proposition~\ref{prop:muproperties}) we see that the weak map part of the definition of CD manifold is satisfied.

The proof of (2) is similar.
\end{proof}

\begin{defn} Any CD manifold arising as in Proposition~\ref{prop:tame} is an affine resp.\ linear {\em CD manifold structure} associated to $\eta$.
\end{defn}

\subsection{The associated complex and the complexes $Y$ and $Z$\label{sec:associated}}

We begin this section with  a corrected version of the definition of the associated complex in~\cite{GM}. See Appendix~\ref{app:oops} for an explanation of the issues with the original version.

\begin{defn}\label{defn:assocposet} Let $X$ be a simplicial manifold. The {\em (affine) associated poset} $Z$ of $X$ is  constructed
	as follows: The elements of Z are 4-tuples $(\Delta, t, y, z)$ satisfying all of the following.
	\begin{itemize}
		\item $\Delta\subset {X^0}$ is a simplex of $X$.
		\item $t$ is an affine coordinate chart at $\Delta$.
		\item $y$ is a rank 2 strong map image of $t$.
		\item $z$ is a nonzero covector of $y$.
	\end{itemize}
	We say  $(\Delta, t, y, z)\leq  (\Delta', t', y', z')$ if the following hold.
	\begin{itemize}
	\item $\Delta\subseteq\Delta'$,
	\item  If $t_{\Delta'}$, $y_{\Delta'}$, and $z_{\Delta'}$ denote the restrictions of $t$, $y$, and $z$ to $\Star(\Delta')^0$, then $t'\leadsto t_{\Delta'}$, $y'\leadsto y_{\Delta'}$, and $z'\geq z_{\Delta'}$.
	\end{itemize}
	
The {\em (affine) associated complex} of $X$ is the order complex $\Cx(Z)$.
	
	If we delete the sign vectors $z$ 
	we get an additional associated poset, which we denote by $Y$, 
	equipped with poset maps $Z\stackrel{\rho}{\to}  Y\stackrel{\pi}{\to}  X$.
\end{defn}

The simplicial complex $\Cx Z$ (modulo correction and rewording) was called the {\em associated complex} in~\cite{GM}. If we replace affine coordinate charts with linear coordinate charts in this definition we get the {\em linear associated poset} and {\em linear associated complex}.

 Let $\mathbb T X$ denote the poset obtained by deleting the oriented matroids $y$ and $z$ in Definition~\ref{defn:assocposet}, and let $\tau:\Cx\mathbb T X\to \Cx X$ be the resulting simplicial map. Then a CD manifold structure on $X$ is just a regular cell decomposition $\hat{X}$ of $\|X\|$ and a poset map $\mu:\hat{X}\to \mathbb T X$ so that the simplicial map $\|\mu\|:\|X\|\to \|\mathbb T X\|$ is a section of $\|\tau\|$. We think of $\mathbb T X$ as a ``fat" CD atlas for $X$ in which all possible CD atlases live, and we think of $Y$ and $Z$ as ``fat" analogs to the spaces $\Y$ and $\Z$ associated to the smooth manifold $M$.
 
 By this analogy, we should hope for the map $\|\rho\|:\|Z\|\to \|Y\|$ to be something similar to a circle bundle. 
 
 \begin{prop}\label{prop:circleqf}  The map 
 $\|\rho\|:\|Z\|\to \|Y\|$ is a  quasifibration with fiber $S^1$.
 \end{prop}

There are two key facts in proving this. The first
 is the Topological Representation Theorem (Theorem~\ref{thm:TRT}), which tells us that the covectors of a rank $r$ oriented matroid $\M$ index the cells in a 
 cell decomposition of $S^{r-1}$. In particular, for each $(\Delta,t,y)\in Y$, $\rho^{-1}(\Delta, t,y)$ is isomorphic to the poset of cells in a regular cell decomposition of $S^1$.
 The second fact is Quillen's Theorem B  (\cite{Quillen}), or more particularly a formulation of this theorem given by Babson in~\cite{Bab}.
 
 \begin{thm}\cite{Bab} Let $f:P\to Q$ be a poset map satisfying both of the following.
 	\begin{enumerate}
 		\item If $q\geq f(p)$ then $\| f^{-1}(q)_{\geq p}\|$ is contractible.
 		\item If $q\leq f(p)$ then $\| f^{-1}(q)_{\leq p}\|$ is contractible. 	\end{enumerate}
Then $\|f\|:\|P\|\to\|Q\|$ is a quasifibration.
\end{thm}

 \begin{proof}  We need to check two criteria for each $(\Delta,t,y,z)\in Z$:
 	\begin{enumerate}
 		\item If $(\Delta', t', y')\geq (\Delta, t, y)$ then $\|\{(\Delta', t', y', z'): (\Delta', t', y', z')\geq (\Delta,t,y,z)\}\|$ is contractible.
%
 		\item If $(\Delta', t', y')\leq (\Delta, t, y)$ then $\|\{(\Delta', t', y', z'): (\Delta', t', y', z')\leq (\Delta,t,y,z)\}\|$ is contractible.
 	\end{enumerate}
 	
 	In either criterion, the question is whether a certain set of covectors $z'$ of $y'$ is contractible. 

 	For the first criterion, the poset $\{z':z'\geq z_{\Delta'}\}$ indexes the cells in an open arc in $S^1$. This arc is nonempty by the definition of weak maps, hence the order complex of $\{z':z'\geq z_{\Delta'}\}$ is contractible. For the second criterion,  the poset $\{z':z\geq z'_\Delta\}$ indexes the cells in a closed arc in $S^1$. This arc is nonempty by Theorem 0.1 in~\cite{weakmapsrep}, hence the order complex of  $\{z':z\geq z'_\Delta\}$  is contractible.
	
 \end{proof}
 
 As shown in Appendix~\ref{app:ss}, from Proposition~\ref{prop:circleqf} we can conclude that $\|\rho\|:\|Z\|\to\|Y\|$ has a well-defined first twisted Chern class $c_1(\|Z\|,\OO)$, which can be described in a similar way as the twisted Chern class $c_1(\Z,\OO)$   as the image under a differential in the Leray-Serre spectral sequence associated to $\|\rho\|$. Further,  Appendix~\ref{app:ss} gives an explicit, simplicial, and local  construction of a cocycle $\Omega$ such that $c_1(\|Z\|,\OO)=[\Omega]$. 
   
 \section{Diagrams of tame maps\label{sec:flags}}
 
 Now we can begin to tie the alternate Chern-Weil formula of Section~\ref{sec:chernweil} to CD manifolds. Our discussion will connect CD manifolds with affine atlases to our formula for $\tilde P(TM\oplus\mathbf 1)$ (which of course coincides with $\tilde P(TM)$). A very similar discussion connects CD manifolds with linear charts to the formula for $\tilde P(TM)$.
 
 Let $\pi:\Y\to M$ be the Grassmannian 2-plane bundle associated to $TM\oplus\mathbf 1$ and $\rho:\Z\to\Y$ the canonical circle bundle over $\Y$. Let $\hat A_\aa:M\to \Gr_\aa(n+1,\R^{X^0})$ be the map of Section~\ref{sec:tameclass} and $(X,\hat X,\Phi)$ an associated CD manifold structure. The goal of this section is to derive from $\hat A_\aa$ a diagram of maps
 
\begin{tikzcd}
\Z\arrow[r]\arrow[d,"\rho"]&\|Z\| \arrow[d,"\|\rho\|"] \\
\Y\arrow[d, "\pi"]\arrow[r]&\|Y\| \arrow[d,"\|\pi\|"]\\
M\arrow[r,"\eta^{-1}"]&\|X\|
\end{tikzcd}

\noindent so that the top square is a map of quasifibrations and the bottom square commutes up to homotopy. 

 \begin{defn} Define the {\em $(r,2)$-flag space} \[\Flag(r, 2;S)=\{(V,W):V\in \Gr(r,\R^S), W\in \Gr(2,V)\}.\] We view $\Flag(r, 2;S)$ as a subspace of $\Gr(r,\R^S)\times  \Gr(2,\R^S)$.
 
Define the {\em $(r,2,1)$-flag space} \[\Flag(r, 2, 1,S)=\{(V,W,\x):V\in \Gr(r,\R^S), W\in \Gr(2,V), \x\in(W-\{0\})/\R_{>0}\}.\] We view $\Flag(r, 2,1;S)$ as a  subspace of  $\Gr(r,\R^S)\times  \Gr(2,\R^S)\times ((\R^S-\{0\})/\R_{>0})$.
 \end{defn}

\begin{defn}  Define the {\em $(r,2)$-flag poset}\[\OMFlag(r,2;S)=\{(\M,\N):\M\in \MacP(r,S), \N\in \Gr(2,\M)\}.\] We view $ \OMFlag(r,2;S)$ as a subposet of $\MacP(r,S)\times\MacP(2,S)$.

Define the {\em $(r,2,1)$-flag poset}\[\OMFlag(r,2,1;S)=\{(\M,\N,\s):\M\in \MacP(r,S), \N\in \Gr(2,\M), \s\in\V^*(\N)-\{0\}\}.\]  We view $ \OMFlag(r,2,1;S)$ as a subposet of $\MacP(r,S)\times\MacP(2,S)\times\{0,+,-\}^S$.
\end{defn}

From the map $\mu:\Gr(r,\R^S)\to\MacP(r,S)$ we get a map $\Flag(r,2,1;S)\to \OMFlag(r,2,1;S)$, which we  also denote by $\mu$, sending $(V,W,\x)$ to $(\mu(V),\mu(W),\sign(\x))$.  Forgetting the last coordinate gives a map $\mu:\Flag(r,2;S)\to \OMFlag(r,2;S)$. These maps are semi-algebraic and upper semi-continuous. Thus the partitions of $\Flag(r, 2;S)$ and $\Flag(r, 2,1;S)$ into preimages can be subdivided to  Whitney stratifications, and any subdivisions of these stratifications into regular cell complexes $\hat F(r,2;S)$ and   $\hat F(r,2,1;S)$ induce poset maps $\hat\mu:\hat F(r,2;S)\to\OMFlag(r,2;S)$ and $\hat F(r,2,1;S)\to  \OMFlag(r,2,1;S)$.

Recall from Proposition~\ref{prop:pullbackaffine} that we can identify $TM\oplus \mathbf 1$ with the pullback of the canonical bundle over $\Gr_\aa(n+1,\R^{X^0})$ by $\hat A_\aa$. Thus we can identify $\Y$ with $\Y'=\{(m,W):m\in M, W\in \Gr(2,\hat A_\aa(m))\}$ and $\Z$ with $\Z'=\{(m,W, \x):m\in M, W\in \Gr(2,\hat A_\aa(m)), \x\in (W-\{0\})/\R_{>0}\}$. From the first of  these identifications we get  a map
\begin{align*}
A_\aa^Y:\Y&\to\|X\|\times\Flag(n+1,2;{X^0})\\
A_\aa^Y(m,W)&=(\eta^{-1}(m), \hat A_\aa(m), W)
\end{align*}

The cell decomposition $\hat X$ of $\|X\|$ is a Whitney stratification, and the product of two Whitney stratifications is a stratification (Theorem~\ref{whitney:product}), so there is a smooth map $\hat A_\aa^Y:\Y\to\|X\|\times\Flag(n+1,2,{X^0})$ arbitrarily close to $A_\aa^Y$ so that $\Y$ gets a Whitney stratification as a pullback. Let $\hat\Y$ be a regular cell decomposition subdividing this stratification. Thus to each cell  $\sigma$ in $\hat\Y$ we have an associated $f^Y(\sigma)=(\Delta, \tau, y)\in Y$.

\begin{prop}
\begin{enumerate}
\item $f^Y:\hat \Y\to Y$ is a poset map.
\item The diagram
\begin{tikzcd}
\Y\arrow[d, "\pi"]\arrow[r,"\| f^Y\|"]&\|Y\|\arrow[d,"\pi"]\\
M\arrow[r,"\eta^{-1}"]&\|X\|
\end{tikzcd}
commutes up to homotopy.
\end{enumerate}
\end{prop}

\begin{proof}That $f^Y$ is a poset map follows from upper semicontinuity of $\mu$.

For each cell $\sigma$ in $\hat\Y$, there is a cell $\tau$ of $\hat X$ such that $\pi(\sigma)\subseteq \tau$ and $f^Y(\sigma)=(\tau,\M,\N)$ for some $\M$ and $\N$. The map $\| f^Y\|$ sends $\sigma$ into $\|\Star(\tau,\M,\N)\|$, and so both  $\pi\circ\| f^Y\|$ and $\eta^{-1}\circ \pi$ send  $\sigma$ into $\|\Star(\tau)\|$. Thus these maps are homotopic by a straight-line homotopy.
\end{proof}

In what follows we let $\eta^Y$ denote $\Cx f^Y$.

Any $\hat A_\aa$ sufficiently close to $A_\aa$ is homotopic to $A_\aa$, so the pulback of the bundle $\Flag(n+1,2,1,\R^{X^0})\to \Flag(n+1,2,\R^{X^0})$ by $\hat A_\aa$ is isomorphic to $\rho: \Z\to\Y$. Let $\rho':\Z'\to \Y$ denote this pullback bundle. 

For each cell $\sigma\in\hat\Y$ and each nonzero covector $\z$ of the rank 2 oriented matroid  associated to $\sigma$, let $[\sigma,\z]=\{(m,V,\x)\in\Z: (m,V)\in\sigma, \x\in (V-\{0\})/\R_{>0},\sign(\x) =\z\}$.

\begin{prop} The set $\hat\Z$ of all nonempty closures of sets $[\sigma,\z]$ is a regular cell decomposition of $\Z$.
\end{prop}

\begin{proof}  For each $(m,W)\in\sigma$ the fiber $\rho^{-1}(m,W)$ is a circle with a cell decomposition indexed by the nonzero covectors $\z$ of the rank 2 oriented matroid $\M$ associated to $\sigma$. In particular,  each $\z$ has rank 0 or 1 in $\V^*(\M)-\{0\}$.

If $\z$ is rank 0 then for each $(m, W)\in\sigma$ there is a unique $(m,W,\x)\in[\sigma,\z]$ such that $\sign(\x)=\z$, and the map $\sigma\to \rho^{-1}(\sigma)$ taking $(m,W)$ to $(m,W,\x)$ is a section of the restriction of $\rho$, and this section extends to the closure of $\sigma$. Thus the closure of $[\sigma,\z]$ is homeomorphic to the closure of $\sigma$.

If $\z$ is rank 1 then there are two rank 0 $\z_1, \z_2\in\V^*(\M)-\{0\}$ such that $\z_1<\z>\z_2$, and $[\sigma,\z]\cap\rho^{-1}(m,W)$ is the 1-cell with boundary points $[\sigma,\z_1]\cap\rho^{-1}(m,W)$ and $[\sigma,\z_2]\cap\rho^{-1}(m,W)$. Thus $[\sigma, \z]\cong \sigma\times (0,1)$ and $\overline{[\sigma,\z]}\cap\rho^{-1}(\sigma)
\cong\sigma\times[0,1]$. (Here $(0,1)$ and $[0,1]$ denote intervals.)  The boundary of $\overline\sigma$ is a regular cell complex, and for each cell $\sigma'$ in this boundary  $\overline{[\sigma,\z]}\cap\rho^{-1}(\sigma')$ is homeomorphic to either $\sigma'\times I$ or $\sigma'$. Thus $\overline{[\sigma,\z]}\cap\rho^{-1}(\overline{\sigma})$ is obtained from $\overline{\sigma}\times I$ by quotients on the boundary that don't change the homeomorphism type.
\end{proof}

Treating $\hat\Y$ and $\hat Z$ as posets of cells, we get simplicial complexes $\Cx\hat\Y$ and $\Cx\hat\Z$ whose geometric realizations are homeomorphic to $\Y$ and $\Z$, and we get a commutative diagram of simplicial maps 

\begin{tikzcd}
\Cx\hat\Z\arrow[r,"\eta^Z"]\arrow[d,"\rho"]&\Cx Z\arrow[d,"\Cx\rho"]\\
\Cx\hat\Y\arrow[r,"\eta^Y"]&\Cx Y
\end{tikzcd}

\noindent which is a map of $S^1$ quasifibrations. Also, we have

\begin{tikzcd}
\Cx\hat\Y\arrow[r,"\eta^Y"]\arrow[d]&\Cx Y\arrow[d,"\Cx\pi"]\\
\Cx\hat X\arrow[r]&\Cx X
\end{tikzcd}

\noindent which commutes up to homotopy.

\section{The combinatorial formula}
 
 Let $Z$ be the affine associated poset of  $X$. As we have seen, associated to $Z$ are a poset $Y$, a homotopy class of maps $\eta^Y:\Y\to\|Y\|$,  and  a simplicial  cocycle
$\Omega$ such that $[\Omega]=c_1(Z,\OO)$.

\begin{thm}\label{thm:combformula} Let $\eta:\|X\|\to M$ be a smoothing of a simplicial manifold $X$ of  dimension $n$. Let $\phi$ be the homology class $[\eta^Y(\Y)]$. 
Then  \[\tilde{p}_i(X)\frown[X]
= (-1)^i \|\pi\|_* ([\Omega]^{n+2i-1} \frown\phi).\]
\end{thm}

\begin{remark}  If we start with $Z$ being the linear associated complex rather than affine, then the formula   is 
 \[\tilde{p}_i(X)\frown[X]
= (-1)^i \|\pi\|_* ([\Omega]^{n+2i-2} \frown\phi)\]
 and the proof  is essentially the same.
 \end{remark}

\begin{proof}
 In the diagram

\begin{tikzcd}
\Z\arrow[r, "\eta^Z"]\arrow[d]&\|Z\| \arrow[d] \\
\Y\arrow[d, "\pi"]\arrow[r, "\eta^Y"]&\|Y\| \arrow[d,"\|\pi\|"]\\
M\arrow[r,"\eta^{-1}"]&\|X\|
\end{tikzcd}

\noindent the top square is a map of quasifibrations with homotopy fiber $S^1$, so that Chern classes pull back, and the bottom square commutes up to homotopy, so that in homology
$\|\pi\|_*\eta^Y_*=\eta^{-1}_*\pi_*$. Also, $\eta^{-1}$ is a homeomorphism. Thus
\begin{align*}
 \tilde{p}_i(X)\frown[X]&=\eta^{-1}_* (\tilde{p}_i(M)\frown[M])\\
 &=\eta^{-1}_*\pi_*((-1)^ic_1(\Z,\OO)^{n-1+2i}\frown[\Y])\qquad\mbox{by Corollary~\ref{cor:smoothform}}\\
 &=(-1)^i\|\pi\|_*\eta^Y_*(c_1(\Z,\OO)^{n-1+2i}\frown[\Y])\\
 &=(-1)^i\|\pi\|_*(c_1(Z,\OO)^{n-1+2i}\frown\phi)\\
&=(-1)^i\|\pi\|_*([\Omega]^{n-1+2i}\frown\phi)\\
\end{align*}

\end{proof}

In some sense the formula of Theorem~\ref{thm:combformula} isn't truly combinatorial. The class $\phi$ is defined in terms of $\Y$, which in turn is defined in terms of the smooth structure on $M$. In the next section we address the conjecture that $\phi$ is  intrinsic to $X$.

\section{Fixing cycles \label{sec:fixing}}

The difficulty in getting to a purely combinatorial formula for the Pontrjagin classes of a simplicial manifold  by Gelfand and MacPherson's approach is the lack of a purely combinatorial analog to the fundamental class of $\Y$. 

\begin{defn} Let $X$ be a simplicial manifold of dimension $n$ with orientation presheaf $\DD$. A {\em fixing cycle} for $X$ is a $(3n -2)$-cycle $\phi\in Z_{3n-2} (Y, \pi^*(\DD))$ such that $\|\pi\|_*(\Omega^{n-1}\frown[\phi])=[X]$.

The fixing cycle is {\em affine} or {\em linear} depending on whether the associated complex is affine or linear.
\end{defn}

The idea is that a fixing cycle should be a combinatorially-defined analog of the fundamental class $[\Y]$.

\begin{prop} The class $\phi=\eta^Y(\Y)$ is an example of a fixing cycle. 
\end{prop}

\begin{proof} $\tilde p_0(M)=1$, so 
\begin{align*}
[X]&=\eta^{-1}_*(\tilde p_0(M))\frown[M])\\
&=\eta^{-1}_*\pi_*(c_1(\Z,\OO)^{n+1-2}\frown [\Y])&\mbox{by Corollary~\ref{cor:chernweil1}}\\
&=\|\pi\|_*(\eta^Y_*(c_1(\Z,\OO)^{n-1}\frown[\Y]))\\
&=\|\pi\|_*(c_1(Z,\OO)^{n-1}\frown \eta^Y_*[\Y])\\
&=\|\pi\|_*(\Omega^{n-1}\frown [\eta^Y_*(\Y)])
\end{align*}
\end{proof}

\begin{conj}\label{conj:fixing} If $X$ is a simplicial manifold then all affine resp. linear fixing cycles are homologous.
\end{conj}

Verification of this conjecture would make the formula of Theorem~\ref{thm:combformula} a true ``combinatorial" formula: in principle, one could calculate a fixing cycle without having a smoothing.

\section{Acknowledgements} 
We thank Cary Malkiewich for helpful conversations.  

\appendix

\section{De Rham and simplicial cohomology}
In the main body of the paper we go from formulas in terms of de Rham cohomology to combinatorial formulas in terms of simplicial cohomology. He we briefly review how one goes between these cohomology theories. 

Most of this discussion is quite standard: see, for instance~\cite{BT} and~\cite{Hatcher} for more detail.

De Rham cohomology is a cohomology with real coefficients associated to smooth manifolds.

\begin{defn}
	A $k$-form on a smooth manifold $M$ is a smooth assignment to each $p \in M$ of an alternating $k$-linear function 
	$$\omega_p : T_pM \times \cdots \times T_pM \rightarrow \mathbb{R}$$
 We denote the space of all $k$-forms on $M$ by $\Omega_{dR}^k(M)$.	
\end{defn}

In local coordinates on an $n$-dimensional smooth manifold, a $k$-form can be expressed as $\omega = \sum_I \omega_I dx^I$, where $I$ ranges over $k$-subsets of $\{1,\ldots, n\}$ and each $\omega_I:M\to\mathbb R$ is a smooth function.
There is a differential map $d:  \Omega_{dR}^k(M) \rightarrow \Omega_{dR}^{k+1}(M)$ that is expressed in terms of local coordinates as 
\[d(\sum_I \omega_I dx^I) = \sum_I d\omega_I \wedge dx^I\]

 The cochain complex $\{\Omega_{dR}^k(M),d\}$ of a compact smooth manifold $M$ is called the {\em de Rham cochain complex}, and the cohomology associated to this cochain complex is the {\em de Rham cohomology}.

Simplicial cohomology can be defined with coefficients in an arbitrary ring $R$. If $X$ is a simplicial complex and $\{v_0, \ldots, v_k\}$ is the vertex set of a simplex $\sigma$ then we define two orderings of $\{v_0, \ldots, v_k\}$ to be equivalent if they differ by an even permutation, and we define an {\em orientation} on $\sigma$ to be one of the two equivalence classes. We define $C_k(X,R)$ to be the module over $R$ generated by the set of oriented $k$-simplicies of $X$, modulo the relation that opposite orientations of a simplex are inverses, and we define  $C^k(X,R)$ to be the dual $\Hom(C_k(X),R)$. There are differentials 
\begin{align*}
d:C_k(X,R)&\to C_{k-1}(X,R)\\
d[x_0,\ldots, x_k]&=\sum_{i=0}^k (-1)^i[x_0,\ldots, \hat x_i,\ldots,x_k]
\end{align*} 
and
\begin{align*}
\partial:C^k(X,R)&\to C^{k+1}(X,R)\\
(\partial\sigma)([x_0,\ldots, x_{k+1}])&=\sum_{i=0}^{k+1} (-1)^i\sigma[x_0,\ldots, \hat x_i,\ldots,x_{k+1}]
\end{align*}
We define the simplicial homology $H_*(X,R)$ and cohomology $H^*(X,R)$  of $X$ with coefficients in $R$  to be the homology and cohomology associated to $\{C_k(X,R), d\}$ and $\{C^k(X,R),\partial\}$. Any two simplicial complexes with the same underlying topological space $M$ have canonically isomorphic homology and cohomology, so we refer to the simplical homology and cohomology of the underlying space.

To see an isomorphism between the de Rham and simplicial cohomology with real coefficients of a smoothly triangulated manifold, we assume our simplicial complex $X$ to be a geometric simplicial complex in Euclidean space. If $h: X\to M$ is a homeomorphism that is smooth on closed simplices, then we define $I:\Omega_{dR}^k(M)\to C^k(X,\R)$ by 
\[ I([\Omega])([\sum_{\sigma}a_{\sigma}\sigma])=\sum_{\sigma}a_{\sigma}\int_{[\sigma]}h^*\Omega\]
This is a map of cochain complexes by Stokes' Theorem, and  it can be shown to induce an isomorphism in cohomology. 

If $\xi:E\to B$ is a real vector bundle and $\omega$ is a form on $E$, we say that $\omega$ has {\em compact vertical support} if, for each compact $K\subseteq M$, $\xi^{-1}(K)\cap\supp(\omega)$ is compact. The set of forms with compact vertical support forms a subcomplex of the de Rham cochain complex of $E$, and the cohomology of this complex is the {\em compact vertical cohomology} $H^*_{cv}(E)$ of $E$. 

Let $\xi:E\to M$ be a real vector bundle over a compact manifold $M$. Let $i:M\to E$ denote the zero section, let $E_0=E-i(M)$, and let $\xi_0:E_0\to M$ be the restriction. Let $S(E)$ be the quotient of $E_0$ by the action of $\R_{>0}$ on fibers. Let $D(E)$ be the fiberwise join of $S(E)$ and $i(M)$ . Then $\xi$ induces a sphere bundle  $S(\xi):S(E)\to M$ and a disk bundle $D(\xi):D(E)\to M$. 
If  $D(\xi)$ is triangulated  then simplicial cohomology $H^*(D(\xi), S(\xi),R)$ is canonically isomorphic to $H^*(E,E_0,R)$ for every ring $R$, and  $H^*(E, E_0,\R)$ is canonically isomorphic to $H^*_{cv}(E)$. In what follows we will identify $H^*(E,E_0,\R)$ with $H^*_{cv}(E)$.

\subsection{Cohomology with various coefficients}\label{app:coefficients}

The particular characteristic classes we will define in terms of de Rham cohomology are each the image of a class defined in integer chomology under the composition
\[ H^i(M,\mathbb Z)\to \Hom(H_i(M,\mathbb Z),\mathbb Z)\hookrightarrow\Hom(H_i(M,\mathbb Z),\mathbb Q)\cong H^i(M,\mathbb Q)
\hookrightarrow\Hom(H_i(M,\mathbb Z),\mathbb R)\cong H^i(M,\mathbb R)\]
Thus these classes exist in $H^i(M,\mathbb Q)$, and we call them {\em rational characteristic classes} (e.g.~ rational Pontrjagin classes). For definitions of the corresponding integer characteristic classes see, for instance, \cite{MS}.

Let $\pi:N\to M$ be a fiber bundle with orientation presheaf $\OO$ and $R$ a ring.  We denote by $H^*(M,\OO)$ the cohomology of $M$ with coefficients in $R$, twisted by the $\OO$. (If $\pi$ is orientable then  
$H^*(M,\OO)$ is canonically isomorphic to $H^*(M,R)$.) The same comments on integer vs.\ rational vs.\ real coefficients hold in the twisted setting, and this justifies some ambiguity in specifying the ring $R$. Most notably, in the context of de Rham cohomology we will define a twisted first Chern class $c_1(E_2,\OO)$ of a vector bundle, and in this setting the ring is $\R$. But because Chern classes are defined with twisted integer coefficients, this class will make sense as a twisted rational cohomology class.

\subsection{Cap product and Poincare duality}

If $[x_0,\ldots,x_{k+l}]$ is an oriented $(k+l)$-simplex and $\psi\in C^{k}(X,R)$ then we define $[x_0,\ldots,x_k]\frown\psi\in C_l(X,R)$ by 
$([x_0,\ldots,x_{k+l}]\frown\psi)=\psi([x_0, \ldots,x_k])[x_k,\ldots, x_l]$. This induces the {\em cap product} $\frown: H_{k+l}(X,R)\times H^k(X,R)\to H_l(X,R)$.
If $X$ is a simplicial complex whose underlying space is a compact orientable $n$-dimensional manifold $M$ then an orientation on $M$ determines an orientation on each maximal simplex of $X$. The sum of all of these oriented simplices, denoted $[M]$, generates $H_n(M,R)$, and Poincare Duality states that the map
$\mbox{PD}:H^{i}(M,R)\to  H_{n-i}(M,R)$ given by $PD(c)=[M]\frown c$ is an isomorphism.

An analog in de Rham chomology to the cap product is given by integration over a smooth submanifold.  Let $[N]$ denote the homology class determined by an oriented compact $(k+l)$-dimensional smooth submanifold $N$ of $M$. If $\alpha$ is a $k$-form on $M$ then we get an element $[N]\frown[\alpha]\in(H^l_{dR}(M))^*$ defined by
\[([N]\frown[\alpha])[\beta]=\int_N \alpha\wedge\beta\]
For homology classes which can be represented by formal sums of submanifolds, this cap product coincides with the simplicial cap product.
 If $M$ itself is  oriented  of dimension $n$,  then Poincare duality  for de Rham cohomology says that the map
\begin{align*}
 \mbox{PD} : H^k_{dR}(M) &\rightarrow (H^{n-k}_{dR}(M))^* \\ 
\mbox{PD}([\omega])([\tau]) &= \int_M \omega \wedge \tau 
\end{align*}
is an isomorphism. The Universal Coefficient Theorem  gives a canonical isomorphism from $(H^*_{dR}(M))^*$ to the (simplicial) homology of $M$, and so we identify $(H^*_{dR}(M))^*$ with $H_*(M)$.

Poincare duality generalizes to non-orientable manifolds if we twist coefficients. Let $X$ be a simplicial complex whose underlying space is a compact $n$-dimensional manifold $M$ with orientation presheaf $\DD$. The same chain-level definition as before gives a cap  product $H_{k+l}(M,\DD)\times H^k(M,R)\to H_l(M,\DD)$. In $H_n(M,\DD)$ we have a fundamental class $[M]$, and cap product with $[M]$ gives an isomorphism $PD: H^i(M,R)\to H_{n-i}(M,\DD)$.
\subsection{Integration over the fiber\label{app:transfer}}

Let $\pi:N\to M$ be a smooth orientable bundle with compact $k$-dimensional fiber $F$. We can integrate an $i$-form $\alpha$ on $E$ over each fiber to get a form on $M$, giving us the {\em transfer map} $\pi^!: H^*(E)\to H^*(M)$. If $i<k$ then $\pi^!([\alpha])=0$. Otherwise $\pi^!([\alpha])\in H_{dR}^{i-k}(M,\DD)$.

\begin{lemma} (Projection formula for fiber bundles)\label{lem:projection} Let $\pi:N\to M$ be a smooth fiber bundle with $M$ and $N$ 
compact. For every form $\alpha$ on $M$ and $\beta$ on $N$,
	\begin{enumerate}
		\item $\pi^![(\pi^*\alpha)\wedge\beta]=[\alpha]\smile\pi^![\beta]$
		\item $\pi_*[\int_N (\pi^*(\alpha)\wedge \beta)]=[\int_M\alpha]\smile \pi^![\beta]$
	\end{enumerate}
\end{lemma}

\begin{prop}\label{prop:starvscap} Let $\pi:N\to M$ be a smooth fiber bundle with $N$ and $M$ 
compact. For any $[\alpha]\in H^*_{dR}(N)$
	\[ \pi_*([\alpha]\frown [N])=\pi^![\alpha]\frown[M].\]
\end{prop}

Proposition~\ref{prop:starvscap} gives us an alternate description of the transfer map in terms of
 Poincare Duality: \[\pi^!=PD^{-1}_M\pi_*PD_N\]

\begin{proof} Let $\alpha$ be a $(k+f)$-form, so that both sides of the equation in the proposition are in $H_{n-k}(M,\DD)$. For each $[\tau]\in H^{n-k}(M,\DD)$,
	\begin{eqnarray*}
		[\tau]\frown( \pi_*([\alpha]\frown [N])&=\pi_*(\pi^*[\tau]\frown([\alpha]\frown[N]))&\mbox{by naturality of cap product}\\
		&=\pi_*((\pi^*[\tau]\smile[\alpha])\frown[N])&\\
		&=\pi_*(\int_N\pi^*(\tau)\wedge\alpha)&\\
		&=[\int_M\tau]\smile\pi^![\alpha]&\mbox{by the Projection Formula}\\
		&=(\tau\smile\pi^![\alpha])\frown[M]&\\
		&=\tau\frown(\pi^![\alpha]\frown[M])
	\end{eqnarray*}
\end{proof}

The {\em (rational) Thom class} of $E$ is the unique $\tau\in H^*_{cv}(E)$ such that the restriction of $\tau$ to a fiber is the preferred generator. The {\em Thom isomorphism} $\Phi:H^k_{dR}(B)\to H^{k+r}_{cv}(E)$ sends each $\alpha$ to $\xi^*(\alpha)\smile\tau$. In simplicial cohomology we can define the Thom class
$\tau\in H^r(E, E_0,R)$ and the Thom isomorphism $\Phi:H^k(B,R)\to H^{k+r}(E, E_0,R)$ in the same way.

\begin{prop} For a real vector bundle $\xi:E\to B$ the transfer map is the inverse of the Thom isomorphism.
\end{prop}

\begin{proof} Let $\alpha\in H^*(B,R)$. Then 
\begin{align*}
\xi^!(\Phi(\alpha))&=\xi^!(\xi^*(\alpha)\smile\tau)\\
&=\alpha\smile\xi^!(\tau)&\mbox{by the Projection Formula}\\
&=\alpha\smile 1\\
&=\alpha
\end{align*}
\end{proof}

\subsection{The Euler class\label{app:euler}}
 The real {\em  Euler class} of an oriented vector bundle $\xi:E\to B$ is defined to be the image of the Thom class under $i^*:H^*_{cv}(E)\to H_{dR}^*(M)$.

 When $M$ is a smooth oriented manifold, the  Euler class of the tangent bundle has a beautiful geometric interpretation (cf. Proposition 12.8 in~\cite{BT}): for any section $s$ of $\xi$ which intersects the zero section  transversely, the  Euler class is the Poincare dual to the homology class given by the zero locus of  $s$ (once appropriate orientation is given to this zero locus). This comes out particularly simply when the dimension of the base space is also $r$, since then the zero locus is a discrete set of points. For each such point $p$, transversality of $s$ and the zero section $z$ tells us that $T_{s(p)}(E)=s_*(T_p(M))\oplus z_*(T_p(M))$. Define $c(p)$ to be 1 if the orientations of the two sides of this equation  agree and $-1$ if they disagree. Then the Poincare dual of  $e(E)$ is
\[e(E)\frown [M]=\sum_pc(p)\]
where the sum is over all points in the zero locus of $s$.

\section{The Chern class with twisted coefficients\label{app:ss}}

For an oriented circle bundle $\rho: \mathcal{Z} \rightarrow \mathcal{Y}$, the Chern class is the obstruction to obtaining a closed global 1-cocycle in $H^1(\mathcal{Z})$ which restricts to a generator of the cohomology in each fiber. A circle bundle that is not necessarily oriented has a twisted Chern class, which lives in cohomology with coefficients in the orientation presheaf. This class can be defined in terms of the Leray-Serre spectral sequence, and this definition can be extended to other maps with similarly-behaved Leray-Serre spectral sequences, not just bundles. In this section we'll lay out a setup that admits this generalization.
Our discussion is adapted from~\cite{BT}.

\begin{defn} A {\em quasifibration} is a surjective map $f: E\to B$ such that, for each $b\in B$ and $e\in f^{-1}(b)$, each induced map $ f_*:\pi_i(E,  f^{-1}(b), e)\to\pi_i(B,b)$ is a bijection.
\end{defn}

It follows that a  quasifibration induces a long exact sequence
\[\cdots\to\pi_{i+1}(B,b)\to\pi_i( f^{-1}(b), e)\to \pi_i(E,e)\to\pi_i(B,b)\cdots\]
A fibrewise map between quasifibrations induces a morphism of long exact sequences.

Every continuous map has an associated fibration (cf.\ \cite{Quillen}, Section 4.3), whose fiber we call the homotopy fiber of the map. If $ f$ is a quasifibration, then each fiber of $ f$ has the same weak homotopy type as the homotopy fiber of $ f$.

We'll be interested in a simplicial quasifibration in which the fiber over each vertex is a simplicial $S^1$ (and therefore the fiber over every point is a homotopy $S^1$).

\begin{prop}\label{prop:QFib}
	Let  $E$ and $B$ be connected simplicial complexes and $f : E \rightarrow B$ be a simplicial  quasifibration with  connected homotopy  fiber $F$. Let $\U$ be the cover of $B$ by open stars, and let $\PPP$ be a presheaf of modules that is locally constant with respect to $\U$.
 Then for each $q$  the presheaf 
 \[\H^q(U)=H^q(f^{-1}U, \PPP(U))\]
  is locally constant.

\end{prop}

\begin{proof}
Let $\tau\subset\sigma$ be simplices of $B$. 
By Whitehead's Theorem with twisted coefficients (\cite{QuillenHA}, Proposition 4, Chapter 2, Section 3), it's enough to show that the inclusion $f^{-1} U_\sigma\hookrightarrow f^{-1} U_\tau$ induces isomorphisms in homotopy groups. We have some immediate observations:
\begin{enumerate}
\item Let $v\in\tau$. Then the inclusions

\begin{tikzcd}
f^{-1}U_\sigma\arrow[rd, hook]\arrow[r, hook]&f^{-1}U_\tau\arrow[d, hook]\\
&f^{-1}U_{\{v\}}
\end{tikzcd}

\noindent commute, so it's enough to consider the case when $\tau=\{v\}$.

\item Let $b$ be a point in the interior of $\sigma$. Then the inclusions

\begin{tikzcd}
f^{-1}b\arrow[rd, hook]\arrow[r, hook]&f^{-1}U_\sigma\arrow[d, hook]\\
&f^{-1}U_{\tau}
\end{tikzcd}

\noindent commute, so it's enough to show that the inclusion $f^{-1}b\hookrightarrow f^{-1}U_\tau$ induces isomorphisms in homotopy groups.

\item $f^{-1}U_{\{v\}}$ retracts to $f^{-1}v$ and $U_{\{v\}}$  retracts to $v$, by straight-line retractions. Thus from the long exact sequence of the triples $(E,f^{-1}U_{\{v\}},f^{-1}v)$ and $(B, U_{\{v\}},v)$ we have that the inclusions $(E,f^{-1} b)\to (E,f^{-1}U_{\{v\}})$ and $(B,\{b\})\to (B,U_{\{v\}})$ induce isomorphisms in homotopy groups.

\end{enumerate}

Consider the commutative diagram

\begin{tikzcd}
(E,f^{-1} b)\arrow[d]\arrow[r, hook]&(E,f^{-1}U_{\{v\}})\arrow[d]\\
(B,\{b\})\arrow[r, hook]&(B,U_{\{v\}})
\end{tikzcd}

	 By our third observation, the horizontal maps induce isomorphisms in homotopy groups, and because $f$ is a quasifibration the first vertical map induces  isomorphisms in homotopy groups as well. Therefore the same holds for the second vertical map.

	 Now consider the long exact sequences in homotopy of the triples $(E, f^{-1}U_{v}, f^{-1}b)$  and $(B,U_{\{v\}},b)$.

	 \begin{tikzcd}
	 \cdots\arrow[r] &\pi_{i+1}(E, f^{-1}b)\arrow{r}\arrow{d}{\cong} & \pi_{i+1}(E, f^{-1}U_{\{v\}})\arrow{r} \arrow{d}{\cong} & \pi_{i}(f^{-1}U_{\{v\}}, f^{-1}b )\arrow{d}\arrow[r]&\cdots\\  
	\cdots \arrow[r]& \pi_{i+1}(B, b)\arrow{r}& \pi_{i+1}(B,U_{\{v\}})\arrow{r}& \pi_{i}(U_{\{v\}}, b)\arrow[r]	&  \cdots
	 \end{tikzcd}

\noindent The first vertical map is an isomorphism because $f$ is a quasifibration, and we have shown the second vertical map to be an isomorphism as well. Thus by the Five Lemma the third vertical map is an isomorphism. But $\pi_{i}(U_{v}, b)=0$, so we conclude that $ \pi_{i}(f^{-1}U_{\{v\}}, f^{-1}b )=0$. Thus the inclusion $f^{-1}b\hookrightarrow f^{-1}U_{\{v\}}$ induces isomorphisms of homotopy groups
\end{proof}

From Proposition~\ref{prop:QFib} we see that for a simplicial $S^1$ quasifibration the presheaf 
$\OO(U):=H^1(f^{-1}U)$ is locally constant with respect to the cover by open stars. $\OO$ plays the same role that the orientation presheaf plays for an $S^1$ fiber bundle. 

The remainder of our discussion, except where noted, holds for any $S^1$ quasifibration $f:E\to B$ equipped with a good cover $\U$ of the base space for which the presheaf $\OO$ defined above is locally constant. 

For each $U\in\U$ let $\H^q(U)=H^q(f^{-1}U,f^*\OO)$. The double complex $K^{p,q}=C^p(f^{-1}\U, \H^q)$ defines a spectral sequence converging to $H^*(E)$. $K^{0,1}$ has a canonical element $\sigma^{0,1}$, generating the cohomology of each fiber. We can describe $\sigma^{0,1}$ explicitly for our two cases of interest.
\begin{enumerate}
\item Let $f:E\to B$ be the circle bundle associated to a smooth rank 2 vector bundle. Then by choosing local coordinates appropriately we can express the connection of the rank 2 bundle locally by a matrix
\[\begin{bmatrix}
0&\omega_{12}\\
-\omega_{12}&0
\end{bmatrix}\]
and $\sigma^{0,1}(U)=[\frac{1}{2\pi}\omega_{12}]$.
\item Let $f:E\to B$ be a simplicial quasifibration such that the fiber over each vertex is $S^1$.
A choice of orientation over a closed star $\overline{U}$ in $B$ determines a 1-cocycle $\Theta$ in $f^{-1}(\overline{U})$ as follows.
For each vertex $v$,  the fiber $f^{-1}(v)$ 
is an $m$-gon for some $m$, and $\Theta$ is given on every $\Theta(\tau)=\frac{1}{m}$ for each positively oriented edge $\tau$ in this fiber. Having fixed this, for each edge $e$ of $B$, $\Theta|f^{-1}(e)$ is the cocycle such that the sum of the squares of the coefficients is minimum. It is rational, since the problem of minimizing a quadratic expression subject to linear constraints  can be solved by linear equations. 
Then for each open star $U$, $\sigma^{0,1}(U)$ is given by the restriction of the cocycle $\Theta$ on $f^{-1}(\overline{U})$.
\end{enumerate}

Because $\H^*$ is locally constant, 
 $\delta\sigma^{0,1}\in K^{1,1}$ is exact, and so 
$\delta \sigma^{0,1} = d\sigma^{1,0}$ for some $\sigma^{1,0}$ in $C^1(f^{-1}\mathcal{U}, \H^{0})$.
Now consider $\delta\sigma^{1,0}\in C^2(f^{-1}\mathcal{U}, \H^0)$.
As $d(\delta \sigma^{1,0}) = \delta(d\sigma^{1,0}) = \delta(\delta \sigma^{0,1}) = 0$, we have that $\delta\sigma^{1,0}$ survives to give a class $-c_1(E,\OO)\in E_2^{2,0}=H^2(B,\OO)$. This class is the image of the $d_2$ differential of $[\sigma^{0,1}]$.
%
%
One easily checks that $\sigma^{0,1}$ extends to a cocycle of $H^*(E,f^*\OO )$ if and only if $c_1(E,\OO)=0$.

\begin{defn} The class $c_1(E,\OO)$ in $H^2(B,\OO)$ is the {\em Chern class} of $f$ (with twisted coefficients). 
\end{defn}

When $f:E\to B$ is the circle bundle associated to a smooth rank 2 vector bundle with curvature given locally by $\begin{bmatrix} 0&\Omega_{12}\\-\Omega_{12}&0\end{bmatrix}$ then $c_1(E,\OO)=[\frac{1}{2\pi}\Omega_{12}]$.

\section{Oriented matroids\label{app:OM}}

An {\em oriented matroid} is a kind of combinatorial data that can be expressed in several equivalent (``cryptomorphic") ways. All of these ways involve the set $\{0,+,-\}$, which is viewed as a poset with unique minimum 0 and maxima $+$ and $-$. For any set $S$, the set  $\{0,+,-\}^S$ is ordered componentwise. Commutative multiplication on $\{0,+,-\}$ is defined in the usual way: $0\cdot x=0$ for all $x$, $+\cdot +=-\cdot-=+$, $+\cdot -=-$. If $\vv=(v_e:e\in E)\in\R^E$ then $\sign(\vv)$ denotes $(\sign(v_e):e\in E)\in\{0,+,-\}^E$.
 
\subsection{Motivation: vector arrangements and subspaces\label{sec:realizOMs}}

Let $(\vv_e:e\in E)$ be an arrangement of column vectors in a rank $r$ vector space $V$ over $\R$ which spans $V$. By choosing coordinates in $V$ and an ordering of $E$ 
we can view this arrangement as an $r\times |E|$ matrix $M$ with columns indexed by $E$. 

Consider the following two ways to extract combinatorial data from $M$.
\begin{enumerate}
\item Let $\chi:E^r\to\{0,+,-\}$ be the function taking each $(e_1,\ldots, e_r)\in E^r$ to the sign of the determinant of the matrix $(\vv_{e_1},\ldots,\vv_{e_r})$. The function $\chi$ is called the {\em chirotope} associated to $M$.  

We are interested in combinatorial data that is independent of the choice of coordinates. If $A\in GL_r$ and $\chi'$ is the chirotope associated to $AM$ then $\chi'=\pm\chi$. So our  interest is in the pair $\{\chi,-\chi\}$. 

\item For each vector $\w=(w_e:e\in E) \in\R^E$, let $\sign(\w)=(\sign(w_e):e\in E)\in\{0,+,-\}^E$. Let $\V^*=\{\sign(\w):\w\in\row(M)\}$. $\V^*$ is the {\em covector set} associated to $M$.

$\V^*$ is coordinate-independent: for any $A\in GL_r$, $M$ and $AM$ have the same associated covector set. 

\end{enumerate}

It is not  hard to see that the pair $\pm\chi$ determines the set $\V^*$ and vice-versa. The {\em oriented matroid} associated to the vector arrangement is the combinatorial information encoded by either $\pm\chi$ or $\V^*$.

Every rank $r$ subspace $V$ of $\R^E$ is the row space of some rank $r$ matrix $M\in\R^{r\times E}$. Viewing $M$ as a vector arrangement, we see that the covector set of the oriented matroid associated to $M$ is $\{\sign(\vv):\vv\in V\}$. Thus it makes sense to also call this the oriented matroid  corresponding to $V$. The set of full-rank matrices with row space $V$ is $GL_rM$; as vector arrangements, the elements of this set are the images of the vector arrangement $M$ under invertible linear transformations.

There are many geometric interpretations of the data encoded by an oriented matroid: we list a few.
\begin{itemize}
\item The pair $\{\chi,-\chi\}$ of chirotopes associated to a vector arrangement tells us
\begin{itemize}
	\item which $(r-1)$-tuples from the arrangement span hyperplanes, and
	\item for each $(r-1)$-tuple spanning a hyperplane $H$, which of the remaining vectors are in $H$ and which pairs among the remaining vectors are in the same open half-space bounded by $H$.
\end{itemize}
\item The covector set associated to $M=(\vv_e:e\in E)$ is $\{\sign(\y M):\y\in\R^r\}$, and the component of $\sign(\y M)$ corresponding to $e\in E$ says whether the vector $\vv_e$ lies on the same side of the hyperplane $\y^\perp$ as $\y$, on the opposite side as $\y$, or on $\y^\perp$.
\item The chirotopes $\pm\chi$  associated to a subspace of $\R^E$ is the pair of signs of Pl\"ucker coordinates for that subspace.
\item The covector set associated to a subspace of $\R^E$ tells us which orthants of $\R^E$ have nonempty intersection with that subspace.
\end{itemize}

\subsection{Formal definitions}
\begin{defn} Let $E$ be a finite set and $r$ a positive integer. A {\em rank $r$ chirotope on elements $E$} is a nonzero alternating function $\chi:E^r\to\{0,+,-\}$ satisfying the following {\em combinatorial Grassmann-Pl\"ucker relations}: for each $e_2, e_3,\ldots, e_r, f_0, f_1, \ldots, f_r\in E$, the set
\[\{(-1)^i\chi(f_i, e_2, \ldots, e_r)\chi(f_0, \ldots, \hat f_i,\ldots, f_r):i\in\{0,\ldots,r\}\}\]
 either is $\{0\}$ or contains both $+$ and $-$. 
\end{defn}

\begin{nota} Let $\x,\y\in\{0,+,-\}^E$. The {\em composition} $\x\circ \y$ is defined to be the element of $\{0,+,-\}^E$ with 
\[\x\circ \y(e)=\begin{cases} 
\x(e)&\mbox{ if $\x(e)\neq 0$}\\
\y(e)&\mbox{ otherwise}
\end{cases}\]
\end{nota}

\begin{defn} Let $E$ be a finite set and $\V^*\subseteq\{0,+,-\}^E$. $\V^*$ is the {\em covector set} of an oriented matroid on elements $E$ if it satisfies all of the following.
\begin{enumerate}
\item $\mathbf 0\in\V^*$.
\item If $\x\in\V^*$ then $-\x\in\V^*$.
\item If $\x,\y\in\V^*$ then $\x\circ \y\in \V^*$.
\item If $\x,\y\in\V^*$, $\x\neq -\y$, and $e\in E$ such that $\x(e)=-\y(e)\neq 0$, then there 
is a $Z\in\V^*$ such that $Z(e)=0$ and, for each $f\in E$,
\begin{itemize}
\item if $\x(f)=\y(f)=0$ then $Z(f)=0$,
\item If $+\in\{\x(f),\y(f)\}\subseteq\{0,+\}$ then $Z(f)=+$, and 
\item If $-\in\{\x(f),\y(f)\}\subseteq\{0,-\}$ then $Z(f)=-$.
\end{itemize}
\end{enumerate}
If $\V^*$ is the  covector set of an oriented matroid, then the {\em rank} of the oriented matroid is the rank of $\V^*$ as a subposet of $\{0,+,-\}^E$.
\end{defn}

A rank $r$ chirotope $\chi\in\{0,+,-\}^{E^r}$ determines a set $\V^*\subseteq\{0,+,-\}^E$ that is the covector set of a rank $r$ oriented matroid. If $\chi$ is a chirotope then $-\chi$ is also a chirotope, and $-\chi$ determines the same set $\V^*$. Conversely, each covector set arises from a unique pair $\pm\chi$ of chirotopes. The correspondence between chirotopes and covector sets is not simple to describe. The information encoded by either a pair $\{\chi,-\chi\}$ of chirotopes or the corresponding covector set is called an oriented matroid.

The motivating examples of Section~\ref{sec:realizOMs} are easily seen to satisfy these definitions. However, there are oriented matroids which are not {\em realizable} as vector arrangements.

\subsection{Some definitions related to oriented matroids}

Many definitions one can associate to vector arrangements generalize to oriented matroids.

\begin{defn} Consider an oriented matroid $\M$ with covector set $\V^*\subseteq\{0,+,-\}^E$ and chirotope $\chi:E^r\to\{0,+,-\}$.
\begin{enumerate}
\item A {\em loop} of $\M$ is an $e\in E$ such that $\x(e)=0$ for each $\x\in\V^*$. 
\item Let $A\cup\{e\}\subseteq E$. We say $e$ is in the {\em convex hull } of $A$ if $\x(e)=+$ for each $\x\in\V^*$ such that $\x(a)=+$ for all $a\in A$.
\item $I\subseteq E$ is {\em independent} if $\chi(e_1,\ldots,x_e)\neq 0$ for some $r$-element superset $\{e_1,\ldots,e_r\}$ of $I$.
\item The {\em rank} of $A\subseteq E$ is the largest size of an independent subset of $A$. 
\end{enumerate}
\end{defn}
If $\M$ arises from a vector arrangement, then the latter three definitions correspond to their usual meaning, while a loop corresponds to a 0 vector. (The name {\em loop} derives from the graph-theoretic interpretation of matroid theory.)

\begin{defn} Let $\M$ and $\N$ be oriented matroids with covector sets $\V^*(\M)\subseteq\{0,+,-\}^E$ and $\V^*(\N)\subseteq\{0,+,-\}^E$. 
\begin{enumerate}
\item We say there is a {\em weak map} from $\M$ to $\N$, written $\M\leadsto\N$, if there are chirotopes $\chi_\M$ for $\M$ and $\chi_\N$ for $\N$ such that $\chi_\N(e_1,\ldots,e_r)\in\{ \chi_\M(e_1, \ldots, e_r),0\}$ for each $e_1,\ldots, e_r$. Equivalently, $\M\leadsto\N$ if for each $\y\in\V^*(\N)$ there is an $\x\in\V^*(M)$ such that $\x\geq \y$.
\item We say there is a {\em strong map} from $\M$ to $\N$, and call $\N$ a {\em strong map image} of $\M$, if $\V^*(\N)\subseteq\V^*(\M)$.
\end{enumerate}
\end{defn}

The motivating example for weak maps is an $\N$ derived from a vector arrangement $\A$ and an $\M$ derived from a vector arrangement obtained from $\A$ by perturbing the vectors slightly.  In terms of subspaces, this corresponds to an $\N$ derived from a subspace $V$ of $\R^E$ such that every open neighborhood of $V$ in $\Gr(r,\R^E)$ contains an element with corresponding oriented matroid $\M$. The motivating example for strong maps is an $\M$ derived from a vector arrangement $\A$ and an $\N$ derived from the image of $\A$ under a linear map. In terms of subspaces, this corresponds to am $\M$ derived from a space $V$ and $\N$ derived from a subspace of $V$.

\subsection{The Topological Representation Theorem}

Let $\V^*$ be the set of covectors arising from an arrangement $(\vv_e:e\in E)$ of vectors in $\R^r$. For simplicity, assume none of the vectors is $\mathbf 0$. Each hyperplane $\vv_e^\perp$ intersects the unit sphere in an equator $S_e^0$, which bounds two open hemispheres $S_e^+$ and $S_e^-$: here $S_e^\epsilon=\{\y:\sign(\vv_e\cdot\y)=\epsilon\}$. These equators partition the unit sphere into cells indexed by $\V^*-\{\mathbf 0\}$: the cell corresponding to $\x\in\V^*-\{\mathbf 0\}$ is $\{\y\in S^{r-1}:\x=(\sign(\y\cdot\vv_e):e\in E)\}$. A remarkable theorem shows that a similar statement is true for non-realizable oriented matroids.

\begin{defn} 
A {\em signed pseudosphere} in $S^{r-1}$ is an ordered partition $(S^0, S^+, S^-)$ of $S^{r-1}$ such that $S^0$ is a codimension 1 topological sphere and $S^0\cup S^+$ and $S^0\cup S^-$ are topological balls. $S^+$ and $S^-$ are called the \textbf{sides} of $S^0$.

An {\em arrangement of signed pseudospheres} in $S^{r-1}$ is a finite collection $((S_e^0, S_e^+, S_e^-):e\in E)$ of signed pseudospheres satisfying all of the following.
	\begin{enumerate}
		\item For every $A\subseteq E$, the set $S_A:=\cap_{e\in A}S_e^0$ is a topological sphere.
		\item  For every $e\in E$ and $A\subseteq E\backslash\{e\}$, either $S_A\subseteq S_e^0$ or $S_A\cap S_e^0$ is a pseudosphere in $S_A$ with sides $S_A\cap S_e^+$ and $S_A\cap S_e^-$.
		\item The intersection of any collection of  sides $S_e^\epsilon$ is either a topological sphere or a topological ball.
	\end{enumerate}
	
	The arrangement is {\em essential} if $S_E=\emptyset$.
\end{defn}


An essential arrangement $\A=((S_e^0, S_e^+, S_e^-):e\in E)$ of signed pseudospheres in $S^{r-1}$ decomposes $S^{r-1}$ into cells, and each cell $\sigma$ can be identified by a sign vector $\x$, where $\x(e)=\epsilon$ if $\sigma\in S_e^\epsilon$. Let $\V^*(\A)$ denote the set of all such sign vectors, together with the vector $\mathbf 0$.

\begin{thm}[Topological Representation Theorem]\cite{FL}\label{thm:TRT} For every essential arrangement $\A$ of signed pseudosperes in $S^{r-1}$, $\V^*(\A)$ is the covector set of a rank $r$ oriented matroid.

Conversely, every rank $r$ oriented matroid arises as $\V^*(\A)$ for some essential arrangement $\A$ of signed pseudosperes in $S^{r-1}$.
\end{thm}

\section{Whitney stratifications\label{app:whitney}}

 If $M$ is a smooth manifold, a {\em Whitney stratification} of a closed subset $T$ of $M$ is a partition of $T$ into manifolds that behaves well with respect to the smooth structure on $M$. Specifically, a Whitney stratification of $T$ is a locally finite partition $(S_i: i\in P)$ of $T$, where each $S_i$ is a locally closed smooth submanifold of $M$ and $P$ is a poset, satisfying
 \begin{enumerate}
 \item For each $\alpha, \beta\in P$,
 \[S_\alpha\cap\overline{S_\beta}\neq\emptyset\Leftrightarrow S_\alpha\subseteq \overline{S_\beta}\Leftrightarrow \alpha\leq\beta.\]
 \item Suppose $\alpha\leq\beta$ in $P$, and suppose $(x_i)$ is a sequence of points in $S_\beta$ converging to some $y\in S_\alpha$. Suppose a sequence $(y_i)$ in $S_\alpha$ also converges to $y$, and suppose that, in local coordinates on $M$, the secant lines $l_i=\overline{x_iy_i}$ converge to some line $l$, and the tangent planes $T_{x_i}S_\beta$ converge to some plane $\tau$. Then
  \begin{enumerate}
  \item\label{2a} {\em (Whitney's condition A)} $T_yS_\alpha\subseteq\tau$ and
  \item\label{2b} {\em (Whitney's condition B)}  $l\subseteq\tau$.
 \end{enumerate}
 \end{enumerate}
(It is not hard to see (cf.~\cite{GorM}, \cite{Trotman}) that Whitney's condition B implies Whitney's condition A.)
 
 We summarize here all we need to know about Whitney stratifications for our purposes.
 
 \begin{thm} (cf. \cite{GorM})  \label{whitney:semialgebraic} Any partition of a subset of $\R^n$ into semialgebraic sets can be subdivided to a Whitney stratification.
 \end{thm}
 	
\begin{thm} (cf.~\cite{GorM}, \cite{Trotman})\label{whitney:opendense} If $M_1$ and $M_2$ are smooth manifolds and $Z_1\subseteq M_1$ and $Z_2\subseteq M_2$ are closed subsets with Whitney stratifications, then $\{f\in C^\infty(M_1, M_2): f|_{Z_1}\mbox{ is transverse to }Z_2\}$ is open and dense in $C^\infty(M_1, M_2)$ (with respect to the Whitney topology).
\end{thm}

(Without getting into the definition of the Whitney topology, we note that when $M_1$ is compact the Whitney topology on $C^\infty(M_1, M_2)$ coincides with the compact-open topology.)

\begin{thm}(\cite{Mather})\label{whitney:pullback} If $M_1$ and $M_2$ are smooth manifolds, $Z_2\subseteq M_2$ is a closed subset with a Whitney stratification, and $f\in C^\infty(M_1, M_2)$ such that $f(M_1)$ is transverse to $Z_2$, then the Whitney stratification pulls back to a Whitney stratification of $f^{-1}(Z_2)$ 
\end{thm}

\begin{thm} (\cite{Trotman}) \label{whitney:product} A product of two Whitney stratifications is a Whitney stratification.
\end{thm}

\begin{thm} (\cite{Gor}) \label{whitney:triangulate} Every Whitney stratification can be subdivided to a triangulation.
 \end{thm}

\section{Errata and departures from~\cite{GM}\label{app:oops}}

Everything in~\cite{GM} is ``morally correct", but in writing this paper we found some minor corrections, as well as things we chose to do differently.
\begin{enumerate}
\item The space $\Y$ in~\cite{GM} was the Grassmannian bundle of $(a-2)$-planes in $E$. We instead make $\Y$  the Grassmannian bundle of $2$-planes in $E$. This makes the relationship to the complex $Y$ come out a bit simpler.
	\item The definition of the associated complex has several problems. 
	\begin{enumerate}
		\item The first is with the partial order implicit in the diagram. In Section 4 of~\cite{GM}, a stratification $\Y_\Delta$ is defined, with a triple $(\Delta,t,y)$  associated to each stratum. If a stratum $S_1$, with associated triple $(\Delta_1,t_1,y_1)$, is contained in the closure of a stratum $S_0$, with associated triple $(\Delta_0,t_0,y_0)$, then the definition of the stratification makes it clear that
	\begin{itemize}
		\item $\Delta_1\supseteq\Delta_0$, and
		\item it need not be true that  $t_0\leadsto t_1$. In fact, if these two oriented matroids have the same nonloops (or ``nonzero elements", in the language of~\cite{GM}) then $t_1\leadsto t_0$, and $t_1$ need not equal $t_0$. 
	\end{itemize}
	This precludes  construction of the continuous map $f$ at the heart of Section 4 in~\cite{GM}. The corrected partial order is in Section~\ref{sec:associated} of the present paper
	\item In~\cite{GM} $z$ is defined to be a  rank 1 oriented matroid, when in fact it should be one of the two nonzero covectors of that oriented matroid.  With their definition the associated complex is analogous to the canonical $\R P^1$ bundle associated to $\Z$, when what we actually want is the double cover of this, analogous to the canonical sphere bundle associated to $\Z$.
	\item Another problem is in the conditions on the relationship between a simplex $\Delta$ and an oriented matroid $t$ associated to $\Delta$ (what we call an affine OM chart associated to $\Delta$). Their definition is too weak: for instance, in a two-dimensional simplicial complex with a simplex $\Delta=\{a, b\}$ and $\mathrm{link}(\Delta)=\{c,d\}$, their definition would allow $t$ to be the oriented matroid of the affine point arrangement shown in Figure~\ref{fig:affineprob}.
	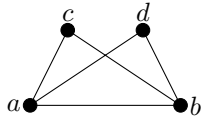
\begin{figure}
	\begin{tikzpicture}
	\draw [fill=black] (0,0) circle (.09cm);
	\draw [fill=black] (2,0) circle (.09cm);
	\draw [fill=black] (.5,1) circle (.09cm);
	\draw [fill=black] (1.5,1) circle (.09cm);
	\draw (0,0)--(2,0)--(1.5,1)--(0,0)--(.5,1)--(2,0);
	\node[left] at (0,0){$a$};
	\node[right] at (2,0){$b$};
	\node[above] at (.5,1){$c$};
	\node[above] at (1.5,1){$d$};
	\end{tikzpicture}
	\caption{Not an affine OM chart\label{fig:affineprob}}
	\end{figure}
This correction isn't important for purposes of the present paper, but it certainly violates the intended idea of $t$.

	\item  \cite{GM} specifies that $t$ has a nonnegative covector. This should really say that $t$ has a positive covector. One trivial reason is because the 0 vector is a covector of every oriented matroid. Less trivially, requiring $t$ to have a positive covector better reflects the idea of oriented matroids arising from smoothings, is combinatorially simpler, and imposes no serious restrictions.
		\end{enumerate}
	\item Proposition 1 of~\cite{GM} should state that the map $\rho$ is a quasifibration. (We prove this in Proposition~\ref{prop:circleqf}.) It need not be a fibration. For instance, let $n=2$, $\Delta=\{a,b,c\}$ be a simplex in $X$, and $t$ be the unique rank 3 oriented matroid whose set of elements is $X^0$ and whose set of nonloops is $\{a,b,c\}$. A topological representation of $t$, with orientations omitted, is given by the solid arcs on the 2-sphere  in Figure~\ref{fig:notafib}. (Note that one solid arc, labelled $a$, coincides with a dotted arc, labelled $y_1$.)

\begin{figure}
	\begin{tikzpicture}
	\draw(0,0) circle (2.5cm);
	\draw(1.77,1.77) to [out=180, in=90] (-1.77,-1.77);
		\node[above right] at (1.77,1.77){$a$, $y_1$};
	\draw (-1.77,1.77) to [out=0, in=90] (1.77,-1.77);
		\node[below right]  at (1.77,-1.77) {$c$};
	\draw (2.5,0) to [out=230, in=-50] (-2.5,0);
		\node[right] at (2.5,0) {$b$};
	\draw [dashed, ultra thick](2.2,1.1) to[out=190, in=60] (-2.2,-1.1) ;
		\node[below left] at (2.2,1.1) {$y_3$};
	\draw [dashed, ultra thick](2.05,1.43) to[out=165, in=95] (-2.05,-1.43) ;
		\node[right] at (2.05,1.43) {$y_2$};
	\draw [dashed, ultra thick] (1.77,1.77) to [out=180, in=90] (-1.77,-1.77);
	\draw [fill=black] (-.2, .55) circle (.09cm);
		\node [below right] at (-.2, .55) {$z_3$};
	\draw [fill=black] (-1.35, -.1) circle (.09cm);
		\node [below right] at (-1.35, -.1) {$z_2$};
	\draw [fill=black] (-1.55, -.75) circle (.09cm);
		\node [ right] at (-1.55, -.75) {$z_1$};
		\end{tikzpicture}
	\caption{The oriented matroids for the counterexample\label{fig:notafib}}
\end{figure}
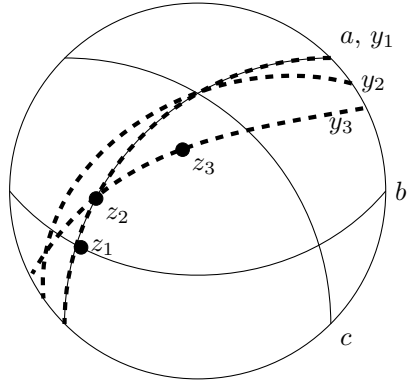

	Let $y_3\leadsto y_2\leadsto y_1$ be the rank 2 contractions of $t$ given by the three dotted arcs. 
Let $z_1$, $z_2$, and $z_3$ be the covectors of $t$ shown.	
Then in $Y$ we have the 2-simplex
	\[\Sigma_Y:=\{(\Delta, t, y_3)\succeq(\Delta, t, y_2)\succeq(\Delta, t, y_1)\}\]
and in $Z$ we have the 3-simplex
\[\Sigma_Z:=\{(\Delta, t, y_3, z_3)\succeq(\Delta, t, y_3, z_2)\succeq(\Delta, t, y_1, z_2)\succeq(\Delta, t, y_1, z_1)\}.\]

Since $z_2$ is not a covector of $y_2$, we see $(\Delta,t,y_2, z_2)\not\in Y$, and so $\Sigma_Z$ is not a face of a simplex that maps under $\|\rho\|$ to $\Sigma_Y$.

Let $\lambda: I\to \|Y\|$ be the linear map sending 0 to $(\Delta, t, y_3)$ and 1 to $(\Delta, t, y_1)$, and let $H:I\times I\to \|Y\|$ be the straight-line homotopy from $\lambda$ to the constant map $I\to(\Delta, t, y_2)$. Then $\lambda$ lifts to the linear map $\lambda'$ from $\Bary(\|\{(\Delta, t, y_3, z_3),(\Delta, t, y_3, z_2)\}\|)$ to $\Bary(\|\{(\Delta, t, y_1, z_2),(\Delta, t, y_1, z_1)\}\|)$.  The image of this lift lies in the interior of $\|\Sigma_Z\|$. Because $\Sigma_Z$ is not a face of a simplex that maps under $\|\rho\|$ to $\Sigma_Y$, we see that the Homotopy Lifting Property fails for the diagram

\begin{tikzcd}
I\times\{0\}\arrow[r,"\lambda"]\arrow[d]&\|Z\|\arrow[d,"\|\rho\|"]\\
I\times I\arrow[r,"H"]&\|Y\|.
\end{tikzcd}
\item A fixing cycle is defined in~\cite{GM} to be an integral cycle, but we could see no need for this.
\item In Theorem 1 of~\cite{GM}, $\frac{1}{2}\Omega$ should be $\Omega$.
	\item In Section 4 of~\cite{GM}, at several places $\mathrm{Cx}$ is written as $C \times$.
%
	\item In Theorem 2 of~\cite{GM}, $u(\delta)$ should be $u(S)$. 
	\item The construction of the fixing cycle in Section 4 of~\cite{GM} has at least one problem. The definition of the orientation $\epsilon$ doesn't work for general stratifications -- for instance, the stratification of the circle into a point and an open interval. 
	Our substitute for their construction of a fixing cycle is in Section~\ref{sec:fixing} of the present paper.
	\item Theorem 1 of~\cite{GM} gives a formula in terms of an arbitrary fixing cycle $\phi$. We were unable to verify that the formula holds for arbitrary fixing cycles. If our Conjecture~\ref{conj:fixing} is correct, then every fixing cycle is homologous to the fixing cycle we used, and so Gelfand and MacPherson's Theorem 1 follows.
\end{enumerate}

\bibliographystyle{amsalpha}
\bibliography{biblio}

\end{document}